\documentclass[a4paper, 12pt]{amsart}   

\usepackage{mathptmx, amssymb,amscd,latexsym, eulervm,calligra,mathrsfs,amsmath,amsthm,amsfonts,amsrefs,tikz, comment, tikz-cd}

\usepackage[colorlinks=true,linkcolor=blue,urlcolor=blue,citecolor=blue]{hyperref}
\usepackage{xurl}
\hypersetup{breaklinks=true}
\usepackage[inner=2.5cm,outer=2.5cm, bottom=3.2cm]{geometry}
\allowdisplaybreaks
\usepackage[nameinlink]{cleveref}

\newtheorem{thm}{Theorem}[section]
\newtheorem{lemma}[thm]{Lemma}

\newtheorem{prop}[thm]{Proposition}

\theoremstyle{definition}
\newtheorem{example}[thm]{Example}
\newtheorem{remark}[thm]{Remark}
\newtheorem{notation}[thm]{Notation}
\newtheorem{definition}[thm]{Definition}

\newtheorem{question}[thm]{Question}
\numberwithin{equation}{section}

\newcommand{\Sym}{\mathrm{Sym}}

\newcommand{\Spec}{\mathrm{Spec}}
\newcommand{\PGL}{\mathrm{PGL}}
\newcommand{\Hilb}{\mathrm{Hilb}}
\newcommand{\Hom}{\mathrm{Hom}}
\newcommand{\Gr}{\mathrm{Gr}}
\renewcommand{\Pr}{\mathrm{Pr}}
\newcommand{\rank}{\mathrm{rank}}

\newcommand{\GL}{\mathrm{GL}}

\newcommand{\kk}{{\mathbf{k}}}

\newcommand{\N}{\mathbb{N}}
\newcommand{\Z}{\mathbb{Z}}

\newcommand{\G}{\mathbb{G}}

\newcommand{\mdeg}{\deg_{T}}

\newcommand{\HF}{\mathrm{HF}}

\newcommand{\sfX}{\mathsf{X}}
\newcommand{\sfY}{\mathsf{Y}}

\newcommand{\sfV}{\mathsf{V}}
\newcommand{\sfW}{\mathsf{W}}

\newcommand{\scF}{\mathscr{F}}
\newcommand{\scS}{\mathscr{S}}
\newcommand{\scQ}{\mathscr{Q}}
\newcommand{\scE}{\mathscr{E}}
\newcommand{\scI}{\mathscr{I}}
\newcommand{\scO}{\mathscr{O}}
\newcommand{\scR}{\mathscr{R}}

\newcommand{\bfd}{\mathbf{d}}
\newcommand{\bfe}{\mathbf{e}}
\newcommand{\bff}{\mathbf{f}}

\newcommand{\bfv}{\mathbf{v}}
\newcommand{\bfw}{\mathbf{w}}

\newcommand{\mfh}{{\mathfrak{h}}}

\newcommand{\mcH}{{\mathcal{H}}}

\newcommand{\mcA}{{\mathcal{A}}}

\newcommand{\mcP}{{\mathcal{P}}}

\newcommand{\ct}{{\mathrm{ct}}}

\newcommand{\al}{\alpha}
\newcommand{\be}{\beta}

\makeatletter
\@namedef{subjclassname@2020}{
  \textup{2020} Mathematics Subject Classification}
\makeatother


\begin{document}

\author[R. Ramkumar]{Ritvik Ramkumar}
\address{Ritvik Ramkumar: Department of Mathematics, University of Notre Dame
\hfill\newline
\indent South Bend, IN, USA}
\email{{\tt rramkuma@nd.edu}}

\author[A. Sammartano]{Alessio Sammartano}
\address{Alessio Sammartano: Dipartimento di Matematica, Politecnico di Milano \hfill\newline
\indent Via Bonardi 9, 20133 Milan, Italy}
\email{{\tt alessio.sammartano@polimi.it}}

\title{Cartwright--Sturmfels Hilbert schemes}

\begin{abstract}
Let $S$ be the Cox ring of a product of $r$ projective spaces. 
In this paper, we study the Cartwright--Sturmfels Hilbert schemes of $S$, which are multigraded Hilbert schemes that only parameterize radical ideals. 
Our main result shows that these Hilbert schemes are always smooth and irreducible if the Picard rank $r$ is at most 2.
This result can be seen as a multigraded analogue of  the famous theorems of Fogarty and Maclagan--Smith, where the Picard rank replaces the dimension of the ambient space.
\end{abstract}

\maketitle

\section{Introduction}\label{SectionIntroduction}

Multigraded rings are ubiquitous in many fields of mathematics, including commutative algebra, algebraic geometry, algebraic combinatorics, representation theory,  and more.
To study deformations of ideals  in a multigraded polynomial ring $S$, 
Haiman and Sturmfels constructed the \emph{multigraded Hilbert scheme} $\mcH^{\mfh}(S)$, 
which parameterizes homogeneous ideals in $S$ with a fixed Hilbert function $\mfh$ \cite{HaimanSturmfels}.
This construction is  a fundamental  and versatile tool, 
which allows  a unifying treatment of various Hilbert schemes 
(Grothendieck Hilbert schemes of projective space, 
Hilbert schemes of points in affine space,
toric Hilbert schemes, 
$G$-Hilbert schemes),
and
can be used in constructing new  moduli spaces (see e.g. \cites{Stable,EE,ST}).
Moreover, it  has many applications outside the field, 
for instance, 
to singularities \cite{Bounded},
complex manifolds \cites{DS,Liu},
number theory \cite{Kreidl},
tensors \cites{BB,CHL,HMV},
and computer vision \cite{AST}.

We remark that
the class of all multigraded Hilbert schemes is in a sense too vast,
given the wide variety of possible gradings and Hilbert functions,
and it presents many pathologies.
For example,
every Grothendieck Hilbert scheme is  connected \cite{Hartshorne}
and has a distinguished point, the lexicographic point, which is  smooth \cite{ReevesStillman};
these facts are  essential, for instance, 
 in the recent classification  of smooth Grothendieck Hilbert schemes \cite{SkjelnesSmith}.
By contrast, multigraded Hilbert schemes can be disconnected \cite{Santos},
 and their lexicographic points can be singular  \cite{RS22}.
However, with suitable assumptions on the grading groups and Hilbert functions, 
one can expect to obtain interesting classes of Hilbert schemes.

We now introduce the class that is the subject of this paper.
Let $\kk$ be an algebraically closed field, and  
$S = \kk[x_{i,j} \, : \, 1 \leq i \leq r, 0 \leq j \leq n_i]$  
the Cox ring of $\mathbb{P}^{\mathbf{n}}=\mathbb{P}^{n_1}\times \cdots \times \mathbb{P}^{n_r}$.
Then, $S$ is a standard multigraded polynomial ring, 
graded by the Picard group $\mathbb{Z}^r$ so that $\deg(x_{i,j}) = \mathbf{e}_i \in \mathbb{Z}^r$.
A remarkable theorem of Brion states that multiplicity-free subvarieties of $\mathbb{P}^{\mathbf{n}}$ are arithmetically  Cohen-Macaulay and normal,
and that they admit a flat degeneration to a reduced union of products of linear spaces \cite{Brion}. 
Cartwright and Sturmfels investigated the diagonal of a product of projective spaces, and  proved that \text{all} flat degenerations are reduced \cite{CartwrightSturmfels}.  
Conca, De Negri, and Gorla developed a general theory of \emph{Cartwright--Sturmfels ideals},
which can be defined as ideals of $S$ with the same Hilbert function as a radical Borel-fixed ideal
\cites{CDNG15,CDNG18,CDNG20}.
They show that this class generalizes the objects of \cite{Brion} and \cite{CartwrightSturmfels},
includes  an abundance of interesting objects from commutative algebra, algebraic geometry and combinatorics,
and that it enjoys many surprising properties related to homological invariants,  Gr\"obner bases, and more. 
The study of Cartwright--Sturmfels ideals and multiplicity-free varieties is a very active area of research,
see for instance \cites{CCRMM,HL} and the references therein. 
 This motivates the following definition:

\begin{definition} A multigraded Hilbert scheme $\mcH^\mfh(S)$ is  \textbf{Cartwright--Sturmfels} if it only parameterizes radical ideals.
\end{definition}

This is equivalent to saying that  $\mcH^\mfh(S)$ parameterizes Cartwright--Sturmfels ideals (see 
\Cref{PropCSpropertyHilbertFunction}), which explains our choice of terminology. 
A notable feature of these Hilbert schemes is that they have a unique Borel-fixed point. 
This, together with  the many rigidity properties of Cartwright--Sturmfels ideals, 
suggests that the class of Cartwright--Sturmfels Hilbert schemes may be well-behaved, and our main theorem 
offers a precise formulation of this idea.

\begin{thm}\label{MainTheorem}
If $S$ is a standard $\Z^2$-graded polynomial ring, 
then every  Cartwright--Sturmfels Hilbert scheme of $S$ is smooth and irreducible.
\end{thm}

Among the $\Z^{r}$-gradings, 
the case  $r = 2$ is  interesting for a specific reason. 
When $r = 1$, it is easy to show that the only Cartwright--Sturmfels Hilbert schemes are the Grassmannians. 
However, when $r \geq 3$, there are examples of singular Cartwright--Sturmfels Hilbert schemes \cite{CartwrightSturmfels}. 
Our result therefore completes the classification of the ranks for which the Cartwright--Sturmfels Hilbert schemes are always smooth. 
Moreover, this situation unexpectedly mirrors
two other notable smoothness results for Hilbert schemes, where the rank of the grading group is replaced with the dimension of the ambient space:

\begin{enumerate}
\item The Grothendieck Hilbert scheme $\Hilb^{p(t)}(X)$, 
where  $X$ is a smooth projective variety  of dimension $r$: for $r = 1$, it is easily shown to be smooth, for $r \geq 3$ it can be singular, while for $r = 2$  it is smooth by a theorem of Fogarty \cite{Fogarty}.
\item The multigraded Hilbert scheme $\mcH^\mfh(\kk[x_1,\dots,x_r])$ of a polynomial ring with arbitrary grading: 
again,
for $r = 1$, it is easily shown to be smooth, for $r \geq 3$ it can be singular, while for $r = 2$  it is smooth by a theorem of Maclagan and Smith \cite{MaclaganSmith}.
\end{enumerate}
We point out that a key distinction between these two results and ours lies in the nature of the ideals being parameterized. 
In the aforementioned results, after standard reductions,  the ideals being parameterized are $0$-dimensional,
whereas
in our case the ideals are  positive-dimensional.

We conclude our paper by outlining potential directions for further investigation on multigraded Hilbert schemes, motivated by both our theorem and its proof.

\section{Multigraded Hilbert schemes}\label{SectionMultigraded}

In this section, we introduce multigraded Hilbert schemes over an arbitrary base, slightly generalizing the original work of Haiman--Sturmfels \cite{HaimanSturmfels}, which dealt with commutative rings, and much of the later research, where the base is a field. 
This broader approach aligns with the literature on Grothendieck Hilbert schemes and offers more flexibility for constructing maps between different Hilbert schemes. 
As we will see in this paper, the expanded setup becomes important even when the main interest is in the case where the base is a field.

Let $Z$ be an arbitrary scheme, and
$\scF_1,\ldots, \scF_p$ be locally free sheaves on $Z$ of finite rank.
Let $\Omega$ be an abelian group, fix elements $\omega_1, \ldots, \omega_p \in \Omega$,
and consider the semigroup morphism $\deg: \N^p \to \Omega$ defined by $\deg(d_1, \ldots, d_p) = \sum_{i=1}^p d_i \omega_i$.

\begin{definition}\label{DefGrading}
We define an \textbf{$\Omega$-grading} on the $\scO_Z$-algebra 
$$
\scR=\Sym(\scF_1\oplus \cdots \oplus \scF_p) \simeq \Sym(\scF_1)\otimes \cdots \otimes \Sym(\scF_p)
$$
by setting  $\deg(\scF_i) = \omega_i$.
More precisely, there is a decomposition
$
\scR=\bigoplus_{\delta\in \Omega} [\scR]_\delta$,
where
$$
[\scR]_\delta := \bigoplus_{\deg(d_1, \ldots, d_p) =\delta} \Sym^{d_1}(\scF_1)\otimes \cdots \otimes \Sym^{d_p}(\scF_p)
$$
is the graded component of degree $\delta$.
The grading is called \textbf{positive} if $\deg^{-1}(0_\Omega) = \{0_{\N^p}\}$.
\end{definition}

Given a morphism $f:Z' \to Z$, the $\scO_{Z'}$-algebra  $\scR_{Z'}:= f^{\star}(\scR)$ has an induced graded decomposition. 
Indeed, since the pullback $f^{\star}$ commutes with direct sums, we have $[\scR_{Z'} ]_\delta :=f^{\star}([ \scR]_\delta)$.

\begin{definition}\label{DefHomogeneousAdmissible}
A sheaf of ideals $\scI \subseteq \scR_{Z'} $ is \textbf{homogeneous} if it respects the graded  decomposition, that is, if 
$\scI=\bigoplus_{\delta\in \Omega} [\scI]_\delta$
where each $[\scI]_\delta \subseteq [\scR_{Z'} ]_\delta$ is a subsheaf. 
It is called \textbf{admissible} if,
for every $\delta \in \Omega$,
 the quotient $[\scR_{Z'}/\scI]_\delta :=[\scR_{Z'}]_\delta/[\scI]_\delta$ is a locally free sheaf on $Z'$ of constant finite rank.
 The \textbf{Hilbert function} of $\scR_{Z'}/\scI$, denoted  by $\HF_\scI: \Omega \to \N$,
 is defined by $\HF_\scI(\delta) := \rank_{Z'} [\scR_{Z'}/\scI]_\delta$.
 \end{definition}

\begin{remark}\label{RemarkAlgebraOverPoint}
When ${Z'}$ is a point of $Z$, say $z\in Z$, 
 the fiber
$\scF_{i,z}$ is a vector space over the residue field $\kappa(z)$,
$\scR_z$ is an $\Omega$-graded polynomial ring over $\kappa(z)$, 
and $\scI_z \subseteq \scR_z$ is a homogeneous ideal with Hilbert function equal to
$\HF_\scI$.
\end{remark}

\begin{definition}\label{DefHilbertFunctor}
Let $\mfh : \Omega \to \N$ be a function.
Let $\mathrm{Sch}_Z$ denote the category of schemes over $Z$, and $\mathrm{Set}$ the category of sets.
The {\bf Hilbert functor} is the functor
$
\underline{\mathrm{Hilb}}_Z^\mfh(\scR):\mathrm{Sch}_Z^\mathrm{op}\to\mathrm{Set}
$ 
defined by the assignment
$$
Z' \mapsto 
\big\{ \text{homogeneous sheaves of ideals } \scI \subseteq \scR_{Z'} \, : \, 
[\scR_{Z'}/\scI]_\delta \text{ is locally free   with rank }   \mfh(\delta)
\big\}.
$$
\end{definition}

\begin{thm}\label{ThmMultigradedHilbertSchemeOverZ}
The functor $\underline{\mathrm{Hilb}}_Z^\mfh(\scR)$ is represented by a 
 scheme, separated and locally of finite type over $Z$,
denoted by $\mcH_Z^\mfh(\scR)$.
If the grading is positive, 
this scheme is proper over $Z$.
\end{thm}

The scheme $\mcH_Z^\mfh(\scR)$ is called the {\bf multigraded Hilbert scheme}.
When $Z$ is an affine scheme, say $Z = \mathrm{Spec}(A)$ for a ring $A$,
this is precisely the object constructed in \cite[Theorem 1.1]{HaimanSturmfels},
and we denote it $\mcH_A^\mfh(\scR)$ for simplicity.

\begin{proof}[Proof of  \Cref{ThmMultigradedHilbertSchemeOverZ}]
Recall that a functor $F:\mathrm{Sch}_Z^\mathrm{op}\to\mathrm{Set}$ is representable 
 if there is a $Z$-scheme $X$ and an isomorphism of functors $F \simeq \Hom_Z(-,X)$.
 The functor  $\underline{\mathrm{Hilb}}_Z^\mfh(\scR)$  clearly satisfies the sheaf property for the Zariski topology.
Let $\{\text{Spec}(A_i)\}_i$ be an affine cover of $Z$.
By \cite[Theorem 1.1]{HaimanSturmfels},  each subfunctor $\underline{\mathrm{Hilb}}_{A_i}^\mfh(\scR_{\text{Spec}(A_i)})$ is representable by a quasi-projective scheme $\mcH_{A_i}^\mfh(\scR_{\text{Spec}(A_i)})$ over $\text{Spec}(A_i)$. 
Then, $\underline{\mathrm{Hilb}}_Z^\mfh(\scR)$ is covered by open representable subfunctors 
and, by \cite[Lemma 26.15.4]{stacks}, it is representable    by a scheme $\mcH_Z^\mfh(\scR)$.
 Since $\mcH_Z^\mfh(\scR)$ is locally quasi-projective over $Z$, it is separated and locally of finite type over $Z$.
 If the grading is positive, then each   $\mcH_{A_i}^\mfh(\scR_{\text{Spec}(A_i)})$
 is projective over $A_i$ by \cite[Theorem 1.1]{HaimanSturmfels}.
It follows that $\mcH_Z^\mfh(\scR)$ is a locally projective $Z$-scheme and thus a proper $Z$-scheme by \cite[01WC]{stacks}.
\end{proof}

\begin{remark}[Base change]\label{RemarkBaseChange}
Let $Z'\to Z$ be a morphism of schemes. 
A standard argument shows that 
\begin{equation}\label{EqBaseChange}
\mcH_{Z'}^\mfh(\scR_{Z'}) \simeq \mcH_Z^\mfh(\scR) \times_Z Z',
\end{equation}
since they represent the same functor.
Equation \eqref{EqBaseChange} is  the base change formula for the multigraded Hilbert scheme.

Applying \eqref{EqBaseChange} when ${Z'}$ is a point of $Z$, say $z\in Z$, 
the fiber of the structure morphism 
$\mcH_Z^\mfh(\scR)\to Z$ over $z$ is naturally isomorphic to the multigraded Hilbert scheme
$\mcH_{\kappa(z)}^\mfh(\scR_z)$.
Thus, 
points of $\mcH_Z^\mfh(\scR)$ correspond bijectively to  homogeneous ideals
$I \subseteq \scR_z$ with Hilbert function  $\mfh$,
where $z \in Z$ is some point 
and $\scR_z$ is the corresponding fiber of $\scR$, cf.  \Cref{RemarkAlgebraOverPoint}.
\end{remark}

\section{Bigraded Hilbert schemes over a product of Grassmannians}\label{SectionBigradedGrassmannians}

In this section, we narrow the focus from the previous section and introduce the notation that will be used throughout the rest of the paper. The main goal of this section is to present several classes of Hilbert schemes and key morphisms between them, which will play a crucial role in the proof of the main theorem.

\begin{notation}[Bigraded polynomial rings]\label{NotationBigraded}
Let $\kk$ be an algebraically closed field.
A {\bf standard  $\Z^2$-graded $\kk$-algebra } is an algebra $A$ with a decomposition $A = \bigoplus_{\bfd\in \Z^2} [A]_\bfd$ where $[A]_{(0,0)} = \kk$ and $A$ is generated over $\kk$ by $[A]_{(1,0)}$ and $[A]_{(0,1)}$.
If $f\in [A]_\bfd$, we write $\deg(f) = \bfd$.
The grading of a standard  $\Z^2$-graded algebra is positive, cf.  \Cref{DefGrading}.

Given non-negative integers $m,n$, we consider the standard $\Z^2$-graded polynomial ring 
$$
S = \kk[x_1, \ldots, x_m, y_1, \ldots, y_n],
$$
where  $\deg(x_i) = (1,0)$ and $\deg(y_i) = (0,1)$.
We have $S = \Sym(\sfX \oplus \sfY) = \Sym(\sfX) \otimes \Sym(\sfY)$, where
$$
\sfX = [S]_{(1,0)} =	\mathrm{Span}_\kk( x_1, \ldots, x_m)
\qquad \text{and} \qquad
\sfY = [S]_{(0,1)} =	\mathrm{Span}_\kk( y_1, \ldots, y_n).
$$

Whenever a tensor product occurs over the base $\kk$, we simply   write $\otimes$ instead of $\otimes_\kk$.
\end{notation}

\begin{notation}[Grassmannians of subspaces]\label{NotationGrassmannians}
Let $0 \leq \al \leq m$ and $0 \leq  \be \leq n$ be integers. 
Let
$$
\G_1= \mathrm{Gr}(\al,\sfX)
\qquad \text{and} \qquad\G_2=\mathrm{Gr}(\be,\sfY)$$
be the Grassmannians 
parametrizing subspaces of $\sfX$ and $\sfY$ of dimension $\al$ and $\be$, respectively,
and denote their product by
$$
\G = \G_1\times \G_2.
$$
On each $\G_i$ there is the tautological sequence of locally free sheaves
\begin{equation}\label{EqTautaologicalSeqGi}
0 \longrightarrow \scS_i \overset{\iota}{\longrightarrow} \scE_i \overset{\pi}{\longrightarrow} \scQ_i \longrightarrow 0
\end{equation}
where, for $\G_1$, $\scE_1 = \sfX \otimes  \scO_{\G_1}$ is the trivial 
sheaf, 
$\scS_1$ is the tautological  subbundle of rank $\al$,
and $\scQ_1$ is the tautological  quotient bundle of rank $m-\al$;
likewise for $\G_2$.
To simplify the notation, 
we drop the index $i = 1,2$ from the natural inclusion $\iota$  and projection $\pi$.
The $\kk$-points of  $\G$ are identified with the pairs $(\sfV,\sfW)$ of subspaces $\sfV \subseteq \sfX, \sfW\subseteq \sfY$ parameterized by $\G$.

Define the $\scO_\G$-algebras
$$
\scR = \Sym(\scE_1) \boxtimes \Sym(\scE_2),
\quad
\scR' = \Sym(\scS_1) \boxtimes \Sym(\scQ_2)
\quad \text{ and } \quad
\scR'' = \Sym(\scQ_1) \boxtimes \Sym(\scS_2),
$$
where $\boxtimes$ denotes a tensor product of pullbacks of sheaves on each factor of $\G$.
These algebras are also  standard $\Z^2$-graded.
Observe that, since $\scE_1, \scE_2$ are trivial sheaves, 
we have $\scR = S \otimes \scO_\G$.
\end{notation}

We will be interested in the multigraded Hilbert schemes 
of these $\scO_\G$-algebras.
For example, 
given a function $\mfh':\Z^2 \to \N$,
one can consider $\mcH_\G^{\mfh'}(\scR')$,
whose $\kk$-points are identified with triples $(I',\sfV,\sfW)$ where
$(\sfV,\sfW)$ is a $\kk$-point of $\G$ 
and $I' \subseteq \Sym(\sfV) \otimes \Sym(\sfY/\sfW)$ is a homogeneous ideal with $\HF_{I'}=\mfh'$.
Similarly, the $\kk$-points of $\mcH_\G^{\mfh''}(\scR'')$ are  triples
$(I'',\sfV,\sfW)$ with 
$I'' \subseteq \Sym(\sfX/\sfV) \otimes \Sym(\sfW)$.
The $\kk$-points of the fiber product 
$\mcH_\G^{\mfh'}(\scR') \times_{\G} \mcH_\G^{\mfh''}(\scR'')$
are   tuples  $(I',I'',\sfV,\sfW)$, with $ I',I'',\sfV,\sfW$ as above.
Finally, the $\kk$-points of $\mcH_\G^{\mfh}(\scR)$ are  triples $(I,\sfV,\sfW)$
where $I\subseteq S$ is a homogeneous ideal with Hilbert function $\HF_I = \mfh$.
We informally say that these Hilbert schemes parameterize the ideals $I', I'', I$, etc. 

Fix  $(\sfV,\sfW) \in \G$ and consider the diagram 
\begin{center}
\begin{tikzcd}
\Sym(\sfV) \otimes \Sym(\sfY/\sfW)  \arrow[d,hook,"\iota"]
& & \Sym(\sfX/\sfV) \otimes \Sym(\sfW) \arrow[d,hook,"\iota"]
\\
\Sym(\sfX) \otimes \Sym(\sfY/\sfW) & 
\arrow[l,swap,"\pi"]
S = \Sym(\sfX)\otimes \Sym(\sfY)
\arrow[r,"\pi"]
& 
\Sym(\sfX/\sfV) \otimes \Sym(\sfY)
\end{tikzcd}
\end{center}
where, again by abuse of notation, we use $\iota $ and $\pi$ to denote the natural maps induced from the inclusions and projections in the tautological sequences \eqref{EqTautaologicalSeqGi}.

\begin{definition}\label{DefConcentratedPositiveDegrees}
Let $A$ be a standard $\Z^2$-graded  $\kk$-algebra
and $I\subseteq A$ be a homogeneous ideal. 
We say that $I$ is \textbf{concentrated in positive degrees} if
$[I]_{(d_1,d_2)} = 0$ whenever $d_1 = 0$ or $d_2 = 0$.
\end{definition}

\begin{lemma}\label{LemmaMapIntersection}
Suppose $\mcH_\G^{\mfh'}(\scR')$ and $\mcH_\G^{\mfh''}(\scR'')$  parameterize ideals concentrated in positive degrees.
 There is a  morphism of $\G$-schemes,
$$
 \mcH_\G^{\mfh'}(\scR') \times_{\G} \mcH_\G^{\mfh''}(\scR'') \to 
\mcH_\G^{\mfh}(\scR),
$$
which, at the level of $\kk$-points, maps $
\big(I',I'',\sfV,\sfW\big) \mapsto \big(I,\sfV,\sfW\big),
$
where
$$
I =\pi^{-1}(\iota(I')) \cap \pi^{-1}(\iota(I'')) \subseteq S
$$ 
and, moreover,
 $\mfh: \Z^2 \to \N$ is a function determined  uniquely  from $\mfh',\mfh''$ by the formula above.
\end{lemma}
\begin{proof}
We first argue at the level of $\kk$-points.
Fix a $\kk$-point 
$$(I',I'',\sfV,\sfW) \in  \mcH_\G^{\mfh'}(\scR') \times_{\G} \mcH_\G^{\mfh''}(\scR'')$$
and
denote  $I_1 :=\pi^{-1}(\iota(I'))\subseteq S$ and $I_2 := \pi^{-1}(\iota(I''))\subseteq S$.

First, we observe that the Hilbert functions of $I_1, I_2$ are uniquely determined by 
$\mfh'$ and $\mfh''$.
Letting $\sfX = \sfV\oplus \sfV'$ for some  subspace $\sfV'\subseteq \sfX$,
we have $\Sym(\sfX) \simeq \Sym(\sfV)\otimes \Sym(\sfV')$
and 
\begin{equation}\label{EqSplitExtensionI}
\frac{\Sym(\sfX)\otimes  \Sym(\sfY/\sfW) }{\iota(I')}\simeq \frac{\Sym(\sfV)\otimes  \Sym(\sfY/\sfW)}{I'}\otimes \Sym(\sfV'),
\end{equation}
thus, the Hilbert function of $\iota(I')$ is uniquely determined by $\mfh'=\HF_{I'}$.
Taking preimage under $\pi$ does not change the quotient,
so the Hilbert function of $\iota(I')$  and $\pi^{-1}(\iota(I'))=I_1$ coincide.
Likewise for $I_2$.

Since $I'  \subseteq \Sym(\sfV) \otimes \Sym(\sfY/\sfW)$ is concentrated in positive degrees, 
it is contained in the ideal $ (\sfV)(\sfY/\sfW)$, that is, the ideal generated by $[\Sym(\sfV)\otimes  \Sym(\sfY/\sfW)]_{(1,1)}$.
In particular, we have $I' \subseteq (\sfV)$.
It follows that $\iota(I') \subseteq (\sfV)$ and $(\sfW) \subseteq I_1 = \pi^{-1}(\iota(I'))\subseteq (\sfV)+(\sfW)$.
Likewise, we have
$(\sfV) \subseteq I_2 \subseteq (\sfV)+(\sfW)$, 
and therefore $I_1 + I_2= (\sfV) + (\sfW)$, so its Hilbert function is independent of $I',I''$.
From the exact sequence
\begin{equation}\label{EqExactSequenceIntersectionI}
0 \to \frac{S}{I_1 \cap I_2} \to \frac{S}{I_1} \oplus \frac{S}{I_2} \to \frac{S}{I_1 + I_2} \to 0
\end{equation}
we deduce that the Hilbert function $\mfh:=\HF_{I_1\cap I_2}$
depends only on $\mfh'$ and $\mfh''$.

These arguments can be reproduced  for more general points.
Let  $f: Z\to \G$ be a $\G$-scheme, and let $\scI'\subseteq \scR'_Z$ and $\scI''\subseteq \scR''_Z$ be admissible homogeneous sheaves of ideals with Hilbert functions $\mfh'$ and $\mfh''$, respectively.
Consider the diagram of $\scO_Z$-algebras induced from the inclusions and projections in the tautological sequences \eqref{EqTautaologicalSeqGi}
\begin{center}
\begin{tikzcd}
\scR'_Z
 \arrow[d,hook,"\iota"]
& & 
\scR''_Z
\arrow[d,hook,"\iota"]
\\
f^\star (\Sym(\scE_1) \boxtimes \Sym(\scQ_2) )
& 
\arrow[l,swap,"\pi"]
\scR_Z
\arrow[r,"\pi"]
& 
f^\star (
\Sym(\scQ_1) \boxtimes \Sym(\scE_2) )
\end{tikzcd}
\end{center}
and define  $\scI_1 := \pi^{-1}(\iota(\scI')) \subseteq\scR_Z$ and $\scI_2=\pi^{-1}(\iota(\scI''))\subseteq\scR_Z$.
Arguing locally on $Z$, 
it follows as in   \eqref{EqSplitExtensionI} that $\scI_1$ and $\scI_2$ are  admissible homogeneous sheaves of ideals with $\HF_{\scI_1}=\HF_{I_1}$ and  $\HF_{\scI_2}=\HF_{I_2}$,
and then it follows as in \eqref{EqExactSequenceIntersectionI}
that  $\scI=\scI_1 \cap \scI_2$ is also an admissible homogeneous sheaf of ideals
with $\HF_{\scI}=\mfh$.
Thus, the assignment $(\scI_1, \scI_2) \to \scI$ defines a natural transformation of  the functors represented by  
$\mcH_\G^{\mfh'}(\scR') \times_{\G} \mcH_\G^{\mfh''}(\scR'') $ and $
\mcH_\G^{\mfh}(\scR)$, 
and, hence,
the desired morphism of $\G$-schemes.
\end{proof}

Applying the base change formula \eqref{EqBaseChange} to  $\scR = S \otimes \scO_\G$,
we see that 
$ \mcH_\G^{\mfh}(\scR)  = \mcH_\kk^{\mfh}(S )\times_\kk \G.
$
So, there is a projection map of $\kk$-schemes $\mathrm{pr} : \mcH_\G^{\mfh}(\scR)  \to \mcH_\kk^{\mfh}(S )$.

\begin{prop}\label{PropProjectionFinite}
Assume that every ideal parameterized by  $\mcH_\G^{\mfh'}(\scR')$ and $\mcH_\G^{\mfh''}(\scR'')$ 
is radical and  concentrated in positive degrees.
Assume, additionally, that for every $\kk$-point 
$$(I',I'',\sfV,\sfW)\in  \mcH_\G^{\mfh'}(\scR') \times_{\G} \mcH_\G^{\mfh''}(\scR'') $$ the ideal $(\sfV)$ is not a minimal prime of $I'$, and the ideal $(\sfW)$ is not a minimal prime of $I''$.
Let $g :  \mcH_\G^{\mfh'}(\scR') \times_{\G} \mcH_\G^{\mfh''}(\scR'') 	\to \mcH_\kk^{\mfh}(S )$ 
denote the composition of the morphism of  \Cref{LemmaMapIntersection}
and 
$\mathrm{pr} : \mcH_\G^{\mfh}(\scR)  \to \mcH_\kk^{\mfh}(S )$.
Then, $g$
is a finite  morphism of $\kk$-schemes.
\end{prop}

\begin{proof}
By \Cref{ThmMultigradedHilbertSchemeOverZ}, all the schemes involved in the statement are proper  over $\kk$, and so $g$ is a proper morphism
by \cite[Lemma 29.41.7]{stacks}.
In order to conclude that $g$ is finite, it suffices to show that it is quasi-finite by \cite[Lemma 30.21.1]{stacks},
and for this it suffices to show that the fiber over a closed point of  $\mcH_\kk^{\mfh}(S)$  is a finite set. 

We follow the notation in the proof of \Cref{LemmaMapIntersection}.
A closed point of  $\mcH_\kk^{\mfh}(S)$  is identified with an ideal $I\subseteq S$.
The fiber is the set of tuples $(I',I'',\sfV,\sfW)\in  \mcH_\G^{\mfh'}(\scR') \times_{\G} \mcH_\G^{\mfh''}(\scR'') $ such that $I = I_1 \cap I_2$.

First, we show that there are finitely many pairs $(I_1,I_2)$ such that $I = I_1 \cap I_2$.
Since $I',I''$ are radical ideals, so are $I_1$ and $I_2$,
and therefore also  $I = I_1 \cap I_2$.
Let $\mathrm{Min}(\cdot)$ denote the set of minimal primes of an ideal.
Since $I = I_1 \cap I_2$,  the set $\mathrm{Min}(I)$ consists of the minimal elements of
$\mathrm{Min}(I_1) \cup \mathrm{Min}(I_2)$.
We claim that $\mathrm{Min}(I)=\mathrm{Min}(I_1) \cup \mathrm{Min}(I_2)$.
Suppose by contradiction, without loss of generality, that there exists $P\in \mathrm{Min}(I_1) \setminus \mathrm{Min}(I)$.
Then, $P$ contains a minimal prime of $I_2$, so $P \supseteq I_2$ and 
thus $P \supseteq I_1+I_2 = (\sfV)+(\sfW) =: Q$.
However, $Q$ is prime and contains $I_1$, so $P = Q$.
Since $(\sfV)+(\sfW)$ is a minimal prime of $I_1 \subseteq \Sym(\sfX)\otimes \Sym(\sfY)$, 
going modulo the prime ideal $(\sfW) \subseteq I_1$ we obtain that $(\sfV)$ is a minimal prime of 
$\iota(I') = \pi^{-1}(I_1) \subseteq \Sym(\sfX) \otimes \Sym(\sfY/\sfW)$.
In turn, since $\iota$ is a polynomial extension, we obtain that $(\sfV)$ is a minimal prime of $I'
\subseteq \Sym(\sfV) \otimes \Sym(\sfY/\sfW) $, contradicting the assumption.

We have thus proved that $\mathrm{Min}(I)=\mathrm{Min}(I_1) \cup \mathrm{Min}(I_2)$.
Since $I_1, I_2$ are radical, they are determined by their sets of minimal primes, and we conclude that there are finitely many possibilities for $I_1$ and $I_2$.

Next, we observe that $\sfV,\sfW$ are uniquely determined by $I_1,I_2$.
Since $I',I''$ are concentrated in positive degrees, we have
$[I']_{(0,1)}=[I'']_{(1,0)}=0$,
so
 $[I_1]_{(0,1)}=\sfW$ and 
$[I_2]_{(1,0)}=\sfV$.

Finally, $I'$ and $ I''$ are determined by $I_1,I_2,\sfV,\sfW$, and the proof is concluded.
\end{proof}

The final goal of this section is showing that, 
for the question of smoothness of the multigraded Hilbert scheme, 
one can restrict to ideals concentrated in positive degrees without loss of generality.

\begin{prop}\label{PropSmoothConcentratedPositiveDegrees}
Let $I\subseteq S$ be a homogeneous ideal, and $\mfh = \HF_I$.
Denote $ \al = \dim_\kk [I]_{(1,0)}$, $\be = \dim_\kk[I]_{(0,1)}$ and $\overline{S}=\kk[x_1,\ldots, x_{m-\al}, y_1, \ldots, y_{n-\be}]$.
If $\mcH_\kk^{{\mfh}}(\overline{S})$ is smooth,
then $\mcH_\kk^\mfh(S)$ is smooth.
\end{prop}
\begin{proof}
Consider the $\scO_\G$-algebra $\overline{\scR} = \Sym(\scQ_1) \boxtimes \Sym(\scQ_2)$.
The Hilbert scheme 
$\mcH_\G^{{\mfh}}(\overline{\scR})$ parameterizes triples $(I,\sfV,\sfW)$
where $I \subseteq \Sym(\sfX/\sfV) \otimes \Sym(\sfY/\sfW)$ and $\HF_I = \mfh$.

Proceeding as in \Cref{LemmaMapIntersection}, 
there is a morphism of $\G$-schemes  $\mcH_\G^{{\mfh}}(\overline{\scR}) \to \mcH^{{\mfh}}_\G(\scR)$
which, at the level of $\kk$-points, maps $(I,\sfV,\sfW)$ to $(\pi^{-1}(I),\sfV,\sfW)$,
where $\pi$ denotes the projection $S \to \Sym(\sfX/\sfV) \otimes \Sym(\sfY/\sfW)$ induced from the tautological sequences \eqref{EqTautaologicalSeqGi}.
As before, 
taking preimage under $\pi$ does not change the quotient, so
$\HF_{\pi^{-1}(I)}=\mfh$.
We then compose this morphism with the projection $\mathrm{pr} : \mcH_\G^{\mfh}(\scR)  \to \mcH_\kk^{\mfh}(S )$ and obtain a morphism of $\kk$-schemes $ \mcH^{{\mfh}}_\G(\overline{\scR}) \to \mcH_\kk^{{\mfh}}(S)$.
We can also construct a morphism  $\mcH_\kk^{{\mfh}}(S) \to \mcH^{{\mfh}}_\G(\overline{\scR})$
sending $$I\mapsto \big(I/([I]_{(1,0)}+[I]_{(0,1)}),[I]_{(1,0)},[I]_{(0,1)}\big).$$
These morphisms are inverse to each other, 
so 
$\mcH_\kk^{{\mfh}}(S) \simeq \mcH^{{\mfh}}_\G(\overline{\scR})$ as $\kk$-schemes.

The algebraic group $\GL = \GL(\sfX) \times \GL(\sfY)$ acts transitively on the $\kk$-points of $\G$,
and this action induces an action on $\overline{\scR}$ and on $\mcH^{{\mfh}}_\G(\overline{\scR})$.
The structure morphism 
$f: \mcH^{{\mfh}}_\G(\overline{\scR}) \to \G$ is clearly $\GL$-equivariant.
Since $\G$ is smooth, it follows that $f$ is flat \cite[Proposition 1.65 (a)]{Milne}.
The fiber of  $f$
 over a closed point $(\sfV,\sfW)$ is 
$ \mcH_\kk^{{\mfh}}(\Sym(\sfX/\sfV) \otimes \Sym(\sfY/\sfW))$
by  \Cref{RemarkBaseChange},
and this scheme is isomorphic to  $\mcH_\kk^{{\mfh}}(\overline{S})$.
If we assume that
 $\mcH_\kk^{{\mfh}}(\overline{S})$ is smooth,
it follows that  $f$ is a smooth morphism,
and thus the total space $\mcH_\kk^{{\mfh}}(S) \simeq \mcH^{{\mfh}}_\G(\overline{\scR})$ is also smooth.
\end{proof}

\section{Cartwright--Sturmfels ideals}\label{SectionCSIdeals}

In this section, we begin by recalling key definitions and results concerning bigraded polynomial rings and Cartwright--Sturmfels ideals, with further details found in \cite{CDNG15}. Following this, we establish a combinatorial framework to study the unique Borel-fixed point of a Cartwright--Sturmfels Hilbert scheme.

The group of graded $\kk$-algebra automorphisms of $S$ is 
$\GL =  \GL(\sfX)\times \GL(\sfY)$.
Its \textbf{Borel subgroup} is $ \mathrm{B}(\sfX)\times \mathrm{B}(\sfY)$,
where $\mathrm{B}(\sfX)\subseteq \GL(\sfX), \mathrm{B}(\sfY)\subseteq \GL(\sfY)$
are the subgroups of automorphisms represented by upper triangular matrices in the fixed bases.
An ideal $I\subseteq S$ is \textbf{Borel-fixed} if it is fixed by the action of the Borel subgroup.
Borel-fixed ideals are monomial ideals, but the converse is not true.
Whether a given monomial ideal is Borel-fixed or not usually depends (also) on the characteristic of $\kk$; 
however, for {radical} monomial ideals there is a combinatorial characterization of the Borel-fixed property.
Recall that a monomial ideal is radical if and only if it is  generated by square-free monomials.
A radical monomial ideal $J \subseteq S$ is Borel-fixed if and only if the following \textbf{exchange properties} hold:
\begin{equation}\label{EqExchangeProperty}
x_h u \in J 
\text{ whenever } 
x_iu \in J, 
\,\,
 h< i,
\qquad
\text{ and }
\qquad
y_h u \in J 
\text{ whenever } 
y_iu \in J, 
\,\,
 h< i.
\end{equation}
It follows that radical Borel-fixed ideals are generated by monomials of degree $(1,0), (0,1)$, or $(1,1)$.

\begin{definition}\label{DefinitionCSIdeal}
A homogeneous ideal $I \subseteq S $ is called {\bf Cartwright--Sturmfels} 
if there exists a radical Borel-fixed ideal $J\subseteq S$ such that $\HF_I=\HF_J$.
\end{definition}

By definition, whether an ideal is Cartwright--Sturmfels depends only on its Hilbert function. 
The next proposition gives several equivalent characterizations of the Hilbert functions with this property.

\begin{prop}\label{PropCSpropertyHilbertFunction}
Let $\mfh$ be a Hilbert function such that $\mcH^{\mfh}_{\kk}(S)$ is non-empty.
Then the following conditions are equivalent:
\begin{enumerate}
\item all the ideals parameterized by $\mcH_\kk^\mfh(S)$ are radical;
\item all the ideals parameterized by $\mcH_\kk^\mfh(S)$ are Cartwright--Sturmfels;
\item some ideals parameterized by $\mcH_\kk^\mfh(S)$ are Cartwright--Sturmfels;
\item $\mcH_\kk^\mfh(S)$ contains  a radical Borel-fixed ideal;
\item $\mcH_\kk^\mfh(S)$ contains  a unique Borel-fixed ideal, and this ideal  is radical.
\end{enumerate}
\end{prop}
\begin{proof} This is a consequence of \protect{\cite[Theorem 3.5]{CDNG15}} and of the well-known fact that every ideal $I \subseteq S$ admits a degeneration to a Borel-fixed ideal.
\end{proof}
\noindent
If $\mcH_\kk^\mfh(S)$ satisfies the  conditions of \Cref{PropCSpropertyHilbertFunction}, we say that it is a {\bf Cartwright--Sturmfels} Hilbert scheme.

An important feature of Cartwright--Sturmfels Hilbert schemes is that every point, possibly after a change of coordinates, degenerates to a common one.
\begin{lemma}\label{LemmaBFDegeneration}
Let $I \subseteq S$ be a Cartwright--Sturmfels ideal, and let $J\subseteq S$ be the unique Borel-fixed ideal such that $\HF_I = \HF_J$.
Then $gI$ admits a one-parameter flat degeneration to $J$ for some $g\in \GL$.
\end{lemma}
\begin{proof}
This follows from the generic initial ideal construction
\cite[Corollary 3.7]{CDNG15}.
\end{proof}

In our setup,  Cartwright--Sturmfels ideals
enjoy the following properties.

\begin{prop}\label{PropBasicPropertiesCS}
Let $I\subseteq S$ be a  Cartwright--Sturmfels ideal.
\begin{enumerate}
\item The minimal generators of $I$ have degree $(1,1)$, $(1,0)$, or $(0,1)$.
\item The Castelnuovo-Mumford regularity of $I$ is at most 2.
\item 
If $L\subseteq S$ is generated by forms of degree $(1,0)$ or $(0,1)$, 
then the ideal 
$(I+L)/L\subseteq {S}/{L}$ is  Cartwright--Sturmfels.
\end{enumerate}
\end{prop}
\begin{proof}
See, respectively, 
\cite[Proposition 2.6, Corollary  2.15, Theorem  2.16 (5)]{CDNG20}.
\end{proof}

By  \Cref{PropCSpropertyHilbertFunction}, 
there is a one-to-one correspondence between Cartwright--Sturmfels Hilbert schemes
 of $S$ parametrizing ideals that are concentrated in positive degrees 
and radical Borel-fixed ideals of $S$ generated in degree $(1,1)$.
In the remainder of the section, we discuss the combinatorics of these ideals.

\begin{notation}[Borel partial order]\label{NotationAntichain}
Given  $p \in \N$, let $[p]=\{1, 2, \ldots, p \}$.
Consider the poset  $\mcP=[m]\times [n]$,  with partial  order  given by componentwise inequality.
Fix an antichain $\mcA \subseteq \mcP$,
$$
\mcA = \big\{(a_1, b_1), \ldots, (a_s, b_s) \big\},
$$
where the elements are ordered so that 
$$
1 \leq a_1 < a_2 < \cdots < a_{s-1} < a_s \leq m
\qquad
\text{and}
\qquad
n \geq b_1 > b_2 > \cdots > b_{s-1} >b_s \geq 1.
$$
We additionally  define 
$
a_0 := 0,\, a_{s+1} := m,\,   b_{0}:=n,\, b_{s+1} := 0.
$

The antichain $\mcA$ determines the following order ideal\footnote{An order ideal $O \subseteq \mcP$ is a  subset such that
$\bfv \in {O}$ whenever  $\bfw \in O$ and $\bfv \leq \bfw$.}
 of $\mcP$
$$
O(\mcA) := \big\{ \bfv \in \mcP \, \mid \, \bfv \leq \bfw \text{ for some } \bfw \in \mcA\big\}.
$$
\end{notation}

\begin{definition}\label{DefIdealAntichain}
Let $\mcA$ be an antichain in $\mcP$, and $O(\mcA)$ its order ideal.
Define the  ideal
\begin{equation}\label{EqIdealFromAntichain}
J_\mcA := \big( x_a y_b \, \mid \, (a,b) \in O(\mcA)\big) \subseteq S.
\end{equation}
\end{definition}

\begin{prop}\label{PropositionBFJiffAntichain}
Let $J\subseteq S$ be an ideal. 
The following conditions are equivalent:
\begin{enumerate}
\item $J$ is  radical,  Borel-fixed, and concentrated in positive degrees;
\item $J = J_\mcA$ for some antichain $\mcA\subseteq \mcP$.
\end{enumerate}
\end{prop}

\begin{proof}
Suppose $J$ is  radical,  Borel-fixed, and concentrated in positive degrees.
By  \Cref{PropBasicPropertiesCS}, $J$ is generated in degree $(1,1)$.
By \eqref{EqExchangeProperty}, the set $O =\{(a,b) \in \mcP \mid x_ay_b \in J\}$ is an order ideal
of $\mcP$. 
Letting $\mcA$ be the set of maximal elements in $O$, we have $O(\mcA) = O$ and $J = J_\mcA$.
Conversely, if $J = J_\mcA$ for some antichain $\mcA\subseteq \mcP$,
then $J$ is radical since it is generated by square-free monomials, 
it is Borel-fixed by  \eqref{EqExchangeProperty}, and it is obviously concentrated in positive degrees.
\end{proof}

\begin{remark}
The monomials indexed by $\mcA$ are sometimes called the {Borel generators} of $J_\mcA$,
see for instance
\cite{BorelGenerators}.
\end{remark}

\begin{example}\label{ExampleBipartiteGraph} 
It is helpful to visualize the  data of the ideal $J_\mcA$ through a bipartite graph,
where 
vertices correspond to the sets $[m]$ and $[n]$, and edges to the set  $O(\mcA)$.

Let $m = 8, n = 7$, $\mcA = \big\{ (1,6), (2,4), (5,1)\big\}$. The corresponding graph is shown below:
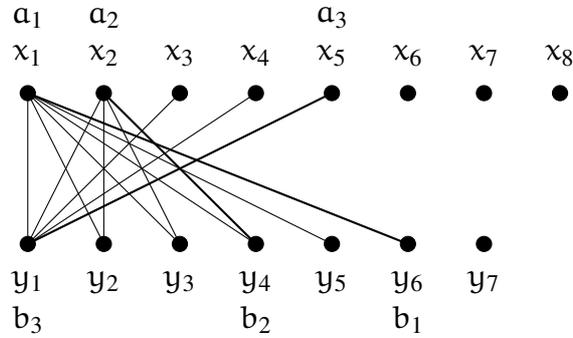
\begin{center}
\begin{tikzpicture}
\draw [   thick] (6,0)--(1,2);
\node at (1,-1) {$b_3$};
\draw [  thick] (4,0)--(2,2);
\draw [  thick] (1,0)--(5,2);
\node at (4,-1) {$b_2$};
\node at (6,-1) {$b_1$};
\node at (5,3) {$a_3$};
\node at (2,3) {$a_2$};
\node at (1,3) {$a_1$};
\foreach \i in {1,...,5}
{
\draw [ thin] (\i,0)--(1,2);
}
\foreach \i in {1,...,3}
{
\draw [ thin] (\i,0)--(2,2);
}
\foreach \i in {3,...,4}
{
\draw [ thin] (1,0)--(\i,2);
}
\foreach \i in {1,...,8}
{
\draw [fill] (\i,2) circle [radius=0.1];
\node at (\i,2.5) {$x_{\i}$};
}
\foreach \i in {1,...,7}
{
\draw [fill] (\i,0) circle [radius=0.1] ;
\node at (\i,-0.5) {$y_{\i}$};
}
\end{tikzpicture}
\end{center}
The edges highlighted with a thicker line correspond to elements of the antichain $\mcA$.
In larger examples, such as 
\Cref{ExampleCutting}, it is convenient to only draw the elements of the antichain. 
For example, in the current example, we obtain the following graph
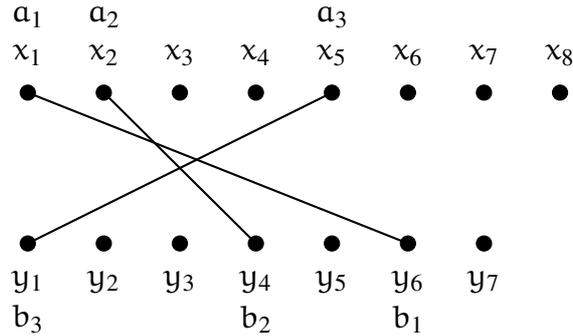
\begin{center}
\begin{tikzpicture}
\draw [  thick] (6,0)--(1,2);
\node at (1,-1) {$b_3$};
\draw [  thick] (4,0)--(2,2);
\draw [  thick] (1,0)--(5,2);
\node at (4,-1) {$b_2$};
\node at (6,-1) {$b_1$};
\node at (5,3) {$a_3$};
\node at (2,3) {$a_2$};
\node at (1,3) {$a_1$};
\foreach \i in {1,...,8}
{
\draw [fill] (\i,2) circle [radius=0.1];
\node at (\i,2.5) {$x_{\i}$};
}
\foreach \i in {1,...,7}
{
\draw [fill] (\i,0) circle [radius=0.1] ;
\node at (\i,-0.5) {$y_{\i}$};
}
\end{tikzpicture}
\end{center}

\end{example}

\section{Tangent space}\label{SectionTangentSpace}

A central goal of this paper is to study the bigraded deformations of Cartwright--Sturmfels ideals. In this section, we determine the weight space decomposition of the tangent space at the unique Borel-fixed ideal, that is, we classify its first-order infinitesimal deformations.

\begin{notation}[Tangent space and fine grading]\label{NotationFineGrading}
Let $\mcA \subseteq \mcP$ be an antichain as in  \Cref{NotationAntichain}, and consider the ideal $J_\mcA \subseteq S$ of  \Cref{DefIdealAntichain}.
Let $\mfh = \HF_{J_\mcA}$.
We denote by $T_\mcA$ the tangent space to the multigraded Hilbert scheme $\mcH^\mfh_\kk(S)$ at the point parametrizing the ideal $J_\mcA$.
By \cite[Proposition  1.6]{HaimanSturmfels},
we have the identification 
\begin{equation}\label{EqTangentSpaceHaimanSturmfels}
T_\mcA = \big[\Hom_S(J_\mcA, S/J_\mcA)\big]_{(0,0)}.
\end{equation}
The ideal  $J_\mcA$ is fixed by the torus of $\GL(\sfX)\times\GL(\sfY)$ consisting of  automorphisms
represented by diagonal matrices in the fixed bases, 
thus, the tangent space $T_\mcA$ is also torus-fixed.
It follows that the vector space $T_\mcA$ is graded with respect to the $\Z^m\oplus \Z^n$-grading on $S$, denoted $\mdeg(-)$, 
defined by
\begin{equation}\label{EqFineGrading}
\Z^m \oplus \Z^n = \langle \bfe_1, \ldots, \bfe_m, \bff_1, \ldots, \bff_n\rangle, \quad
\mdeg(x_i) = \bfe_i,
\quad
 \mdeg(y_j) = \bff_j.
\end{equation}
In this section,
we consider both the $\Z^2$-grading $\deg(-)$ of 	\Cref{NotationBigraded}
 and the $\Z^m\oplus\Z^n$-grading $\deg_T(-)$ of \eqref{EqFineGrading}.
The latter is a refinement of the former, by the  homomorphism $\Z^m \oplus \Z^n \to \Z^2$ that maps
$\bfe_i \mapsto (1,0), \bff_j \mapsto (0,1)$.
To avoid ambiguity, 
in this section alone,
we use the words ``multigraded'' and ``multidegree'' when referring to $\deg_T(-)$, and ``bigraded'' and ``bidegree'' when referring to $\deg(-)$.
(In the other sections, we only consider the $\Z^2$-grading, thus, we simply write ``graded'' and ``degree'').
 \end{notation}

Our goal is to determine the multigraded decomposition of $T_\mcA$.
We show, in particular,  that this decomposition is multiplicity-free, 
that is, every non-zero multigraded component is one-dimensional.

We begin by introducing a classification of tangent vectors  by their multidegree.

\begin{definition}
Let $\varphi\in T_\mcA$ be a multigraded tangent vector.
 We say that $\varphi$ is
\begin{enumerate}
\item a {\bf linear tangent} if  $\mdeg(\varphi) = \bfe_j - \bfe_i$ 
or $\mdeg(\varphi) = \bff_h - \bff_k$
for some $i \ne j$ or $h \ne k$;

\item  a {\bf quadratic tangent} if
 $\mdeg(\varphi) = \bfe_j - \bfe_i + \bff_h - \bff_k$ for some $i \ne j, h \ne k$.
\end{enumerate}
\end{definition}

\begin{remark}\label{RemarkTangentVectors}
Let $\varphi\in T_\mcA$ be a non-zero multigraded tangent vector.
\begin{enumerate}
\item Since $\deg(\varphi) = (0,0)$, 
$\varphi$ maps a monomial minimal generator of $J_\mcA$ to a constant multiple of a monomial of $[S/J_\mcA]_{(1,1)}$.
Since 
$J_\mcA$ is generated 
by  monomials $x_iy_k$,
it follows that $\varphi$ is either  linear or  quadratic.
\item Assume $\varphi$ is linear, with $\mdeg(\varphi) = \bfe_j - \bfe_i$.
Then, up to scaling coefficients, $\varphi(x_iy_k) = x_j y_k$ for some $k$ such that 
$(i,k) \in O(\mcA)$,
$(j,k) \notin O(\mcA)$.
By \eqref{EqExchangeProperty}, this implies that $i < j$.
Moreover,  $\varphi(x_{i'}y_h) = 0$ whenever $i' \ne i$.
An analogous statement holds if $\mdeg(\varphi) = \bff_h - \bff_k$.
\item Assume $\varphi$ is quadratic, 
with  $\mdeg(\varphi) = \bfe_j - \bfe_i + \bff_h - \bff_k$.
Then, up to scaling coefficients, $\varphi(x_iy_k) = x_jy_h$, 
 $(i,k) \in O(\mcA), 
(j,h) \notin O(\mcA)$,
and $\varphi(x_{i'}y_{k'}) =0$ for all $(i',k')\ne (i,k)$.
\end{enumerate}
\end{remark}

The following  identification will be useful for classifying  tangent vectors.

\begin{lemma}\label{LemmaTangentKillSyzygy}
There is an isomorphism of multigraded vector spaces between $T_\mcA$ 
and the subspace of $\Hom_\kk\big([J_\mcA]_{(1,1)},[S/J_\mcA]_{(1,1)}\big)$
spanned by  
the  maps
$\varphi : [J_\mcA]_{(1,1)} \to [S/J_\mcA]_{(1,1)}$ such that 
$$
x_j\varphi(x_i y_k) = x_i\varphi(x_j y_k) 
\quad
\text{whenever}
\quad
(i,k), (j,k)\in O(\mcA)
$$
and
$$
y_h\varphi(x_i y_k) = y_k\varphi(x_i y_h) 
\quad
\text{whenever}
\quad
(i,h), (i,k)\in O(\mcA).
$$
\end{lemma}

\begin{proof}
Recall that $J_\mcA$ is generated 
by the monomials $x_iy_h$ with $(i,h) \in O(\mcA)$.
By \Cref{PropBasicPropertiesCS}, the Castelnuovo-Mumford regularity of $J_\mcA$ is equal to 2, 
hence, the first syzygies of $J_\mcA$ are generated by the relations
$x_i(x_j y_k) = x_j(x_i y_k)$ with $(i,k), (j,k)\in O(\mcA),$
and
$y_h(x_i y_k) = y_k(x_i y_h) $ with $ (i,h), (i,k)\in O(\mcA).$
Since 
a $\kk$-linear map $\varphi : [J_\mcA]_{(1,1)} \to [S/J_\mcA]_{(1,1)}$ 
extends to an $S$-linear map $\varphi: J_\mcA \to S/J_\mcA$ if and only if 
it is compatible with the first syzygies of $J_\mcA$, we obtain the desired conclusion.
\end{proof}

By \Cref{LemmaTangentKillSyzygy},
a tangent vector can be interpreted pictorially as a map from the set of edges 
to the set of non-edges in the graph determined by $O(\mcA)$, cf. \Cref{ExampleBipartiteGraph}.

The following  are the main results of this section.

\begin{prop}[Description of linear tangents]\label{PropLinearTangents}
Let $\bfd = \bfe_j - \bfe_i$ with $i \ne j$.
Then,   $\dim_\kk [T_\mcA]_\bfd \leq 1$.
We have $\dim_\kk [T_\mcA]_\bfd = 1$ if and only if
$i < j$ and there exists $k\in [n]$ such that $(i,k)\in O(\mcA)$ and $(j, k)\notin O(\mcA)$,
equivalently,
if and only if there exists $\ell$ such that $i \leq a_\ell < j$.
The analogous results hold for $\bfd = \bff_h - \bff_k$.
\end{prop}

\begin{proof}
Assume $0 \ne \varphi\in T_\mcA$ is a multigraded tangent vector with $\mdeg(\varphi) = \bfe_j - \bfe_i$.
By  \Cref{RemarkTangentVectors}, we have $i<j$,
 there exists a $k$  such that $\varphi(x_i y_{k}) = cx_j y_{k} \ne 0$, 
 with 
 $c \in \kk$, 
 $(i,k)\in O(\mcA)$ and $(j,k)\notin O(\mcA)$,
 and $\varphi(x_{i'}y_h) = 0$ whenever $i' \ne i$.
We also have  $\varphi(x_iy_h) = x_j y_h = 0$ whenever $(j,h)\in O(\mcA)$.
We may assume, without loss of generality, that $c=1$.
It remains to determine $\varphi(x_iy_h)$ when $(i,h)\in O(\mcA)$  and $(j,h)\notin O(\mcA)$.
In this case, we have 
$$ 
y_{k} \varphi(x_iy_{h} ) =
y_{h}\varphi(x_i y_{k}) =
x_j y_{h}  y_{k} \ne 0
$$
since $(j,h),(j,k) \notin O(\mcA)$ and $J_\mcA$ is generated in bidegree (1,1).
It follows that
$
 \varphi(x_iy_{h} ) =
 x_j y_{h}.
$
This concludes the description of  $\varphi$,
and shows that $\dim_\kk [T_\mcA]_\bfd \leq 1$.

Now,  assume that  $i < j$ and that  $(i,k)\in O(\mcA),(j, k)\notin O(\mcA)$ for some $k \in [n]$.
Define a $\kk$-linear map $\varphi : [J_\mcA]_{(1,1)} \to [S/J_\mcA]_{(1,1)}$  as in the previous paragraph.
We claim that $\varphi$ satisfies the conditions of \Cref{LemmaTangentKillSyzygy}.
We must check the compatibility conditions for $\varphi$ on pairs of monomials corresponding to adjacent edges in $O(\mcA)$, and, of course, we may assume that $\varphi$ is non-zero on at least one of them. 
Thus, let $h \in [n]$ be such that $(i,h)\in O(\mcA), (j,h)\notin O(\mcA)$, so $\varphi(x_iy_h) = x_j y_h \ne 0$.
If $(i',h)\in O(\mcA)$ with $i' \ne i$, then $x_{i'}y_h =0 $ in $S/J_\mcA$, and $\varphi(x_{i'}y_h) =0$ by definition of $\varphi$,
thus,
$$
x_{i'}\varphi(x_iy_h) = x_{i'}  x_j y_h = x_j \cdot 0  = x_i \cdot 0 = x_i \varphi(x_{i'} y_h). $$ 
If $(i,h')\in O(\mcA)$ with $h \ne h'$, then, by definition of $\varphi$, we have
 $$
 y_{h'}\varphi(x_iy_h) =   x_j y_h y_{h'}= y_h \varphi(x_i y_{h'}). $$
Note that this relation holds whether or not $(j, h') \in O(\mcA)$, in which case all terms vanish.
By \Cref{LemmaTangentKillSyzygy}, we deduce that $[T_\mcA]_\bfd \ne 0$, as desired.

Finally, the conditions $i < j$ and the existence of a $k \in [n]$ such that $(i, k) \in O(\mcA)$ and $(j, k) \notin O(\mcA)$ are equivalent to having an $\ell$ such that $i \leq a_\ell < j$. To see the equivalence, note that for the forward direction, one can choose any $\ell$ where $(a_\ell, b_\ell) \geq (i, k)$; for the reverse direction, simply set $k = b_\ell$.
\end{proof}

\begin{prop}[Description of quadratic tangents]\label{PropQuadraticTangents}
Let $\bfd 
= \bfe_j - \bfe_i + \bff_h - \bff_k$ with $i \ne j, h \ne k$.
Then, $\dim_\kk [T_\mcA]_\bfd \leq 1$.
We have $\dim_\kk [T_\mcA]_\bfd = 1$ if and only if there exists $1 \leq\ell\leq s$  such that 
\begin{enumerate}
\item  $(i,k) = (a_\ell, b_\ell)$,
\item $a_{\ell-1}= a_{\ell}-1$,
\item $b_{\ell+1}= b_{\ell}-1$,
\item $a_\ell+1 \leq j \leq a_{\ell+1}$, 
\item $b_\ell+1 \leq h \leq b_{\ell-1}$.
\end{enumerate}
\end{prop}

\begin{proof}
Assume  $0 \ne\varphi\in T_\mcA$ is a multigraded tangent vector with $\mdeg(\varphi) = 
 \bfe_j - \bfe_i + \bff_h - \bff_k$.
 By \Cref{RemarkTangentVectors}, 
 we have 
 $\varphi(x_iy_k) = x_jy_h$, 
 $(i,k) \in O(\mcA), 
(j,h) \notin O(\mcA)$, and
$\varphi(x_{i'}y_{k'}) =0$ for all $(i',k')\ne (i,k)$. 
In particular, $\varphi$ is uniquely determined up to scalars, so  $\dim_\kk [T_\mcA]_\bfd \leq 1$.

First, we claim that $i < j$ and $k < h$.
If  this is not the case, then, without loss of generality, we have $i \geq j$,
in fact, $i > j$ since  $i \ne j$.
By \eqref{EqExchangeProperty}, we obtain $x_jy_k \in J_\mcA$.
Then,
$
x_i \varphi(x_j y_k)  = x_j \varphi(x_i y_k)  =  x_j^2 y_h \ne 0 
$
since $(j,h) \notin O(\mcA)$ and $J_\mcA$ is radical.
This contradicts the fact that $\varphi(x_jy_k)=0$.
 
Next, we show that $(i,k) = (a_\ell, b_\ell)$ for some $\ell\in [s]$.
Assume by contradiction this is not the case.
Since every element of $O(\mcA)$ is dominated by some element of $\mcA$,
without loss of generality
we have $i < a_\ell, k \leq b_\ell$ for some $\ell\in [s]$.
Then,
$
x_i \varphi(x_{i+1} y_k) =   x_{i+1} \varphi(x_i y_k)  =  x_{i+1} x_j y_h.
$
We claim that $x_{i+1} x_j y_h$ is non-zero in $ S/J_{\mcA}$,
yielding a contradiction since 
$\varphi(x_{i+1}y_k)=0$.
We have $(j,h) \notin O(\mcA)$, thus,
since $J_\mcA$ is generated in bidegree $(1,1)$, it suffices to show that  $(i+1, h)\notin O(\mcA)$.
Suppose $(i+1, h)\in O(\mcA)$,
then $(i, h)\in O(\mcA)$ by \eqref{EqExchangeProperty}, 
and
$
y_k \varphi(x_i y_h) =   y_h \varphi(x_i y_k)  =  x_j y_h^2 \ne 0 
$
since $(j,h) \notin O(\mcA)$ and $J_\mcA$ is radical.
This contradicts the fact that $\varphi(x_i y_h) = 0$,
and concludes the proof of statement (1).

Next, we claim that $(i',h) \in O(\mcA)$ whenever $i'<i$, and $(j,k')\in O(\mcA)$ whenever $k'<k$.
It suffices to prove the first one.
Assume by contradiction that $(i',h) \notin O(\mcA)$.
Observe that $(i',k) \in O(\mcA)$ by \eqref{EqExchangeProperty}.
Then,
$
x_i\varphi(x_{i'}y_k) =
x_{i'}\varphi(x_iy_k) =x_{i'} x_j y_h \ne 0
$,
since $(i',h),(j,h) \notin O(\mcA)$ and $J_\mcA$ is generated in bidegree $(1,1)$.
This contradicts $\varphi(x_{i'}y_k)=0$.

We now prove the remaining four statements.
Suppose $i=a_\ell > 1$.
By the previous paragraph we have  $(a_\ell-1,h) =(i-1,h)\in O(\mcA)$, 
whereas $(a_\ell,h)=(i,h) \notin O(\mcA)$ since $h > k = b_\ell$.
We  deduce  that $(a_\ell-1,h')\in \mcA$ for some $h'\geq h $, so, 
$(a_\ell-1,h') = (a_{\ell-1}, b_{\ell-1})$.
This shows (2) and (5), which also hold when  $a_\ell = 1$, so $\ell= 1$, cf. Notation \ref{NotationAntichain}.
The same arguments show (3) and (4).
This concludes the  proof that the conditions of the proposition
are necessary for having $\dim_\kk[T_\mcA]_\bfd=1$.

Conversely, we show that the conditions are  sufficient.
Define a $\kk$-linear map $\varphi : [J_\mcA]_{(1,1)} \to [S/J_\mcA]_{(1,1)}$  
by
 $\varphi(x_iy_k) = x_jy_h$
 and
$\varphi(x_{i'}y_{k'}) =0$ for all $(i',k')\ne (i,k)$. 
We check that $\varphi$ satisfies the conditions of \Cref{LemmaTangentKillSyzygy}.
Since $\varphi$ is non-zero only on the monomial $x_iy_k$,
this amounts to showing that 
$x_{i'}x_jy_h\in J_\mcA$ if $(i',k) \in O(\mcA)$ and $i' \ne i$, 
and $y_{k'}x_jy_h\in J_\mcA$ if $(i,k') \in O(\mcA)$ and $k'\ne k$.
Let $i' \ne i$ such that $(i',k) \in O(\mcA)$.
Since $(i,k) = (a_\ell,b_\ell) \in \mcA$ by (1), it follows that $i'<i$,
thus, 
$i' \leq a_{\ell}-1= a_{\ell-1}$ by (2).
Since $h \leq b_{\ell-1}$ by (5), 
we obtain that  $(i',h)\in O(\mcA)$, so $x_{i'}x_jy_h\in J_\mcA$.
The other case is analogous.
\end{proof}

\section{Cutting process and proof of the main theorem}\label{SectionCuttingProof}

This section is dedicated to proving \Cref{MainTheorem}, which we accomplish through an intricate induction process with intermediate steps parameterized by various Hilbert schemes. To clarify the argument, we introduce the notion of a  \emph{cutting process} for the combinatorial data from \Cref{SectionCSIdeals}. 
This cutting process is crucial to each induction step, 
and makes the overall argument more streamlined.

Throughout this section, we continue to follow  \Cref{NotationAntichain}. 
Thus,
$\mathcal{A} = \{(a_1,b_1), \ldots, (a_s,b_s)\} 
$ denotes an antichain of the poset $\mcP=  [m]\times [n]$, where 
\begin{align*}
0 &= a_0 < a_1 < a_2 < \cdots < a_{s-1} < a_s \leq a_{s+1} = m,\\
n &= b_0 \geq b_1 > b_2 > \cdots > b_{s-1} > b_s > b_{s+1} = 0.
\end{align*}
We assume that  $\mathcal{A}\ne \emptyset$, equivalently, 
that $s>0$.

\begin{definition}\label{DefCuttingThreshold}
The {\bf cutting threshold} of $\mathcal{A}$,
denoted by $\mathrm{ct}(\mathcal{A})$, is  the smallest   $q\in \N$  such that
$0 \leq q \leq s$ and
 the following conditions hold for all  $ q+1\leq i \leq s$:
\begin{enumerate}
\item 
$a_{i} < m\,\,$ and $\,\,b_{i} < n$;
\item $a_{i}=a_{i-1}+1\,\,$ and $\,\, b_{i} = b_{i+1} +1$.
\end{enumerate}
\end{definition}

Let us  comment briefly on this definition.
We have $\ct(\mcA)=q<s$ if the   integers $a_{q},\ldots, a_s$ are consecutive and smaller than  $m$, and the  integers $b_s,  \ldots, b_{q+1}$ are consecutive, starting at $1$, and smaller than $n$;
moreover, $q$ is the smallest number for which these conditions hold.
We have $\ct(\mcA)=s$ when one of the following conditions occurs: 
$a_s = m$, 
 $b_s  =n$ (and $s = 1$), 
 $a_s > a_{s-1} +1$, 
or $b_s > 1$.

\begin{definition}[Cutting process]\label{DefCuttingProcess}
Let $q = \ct(\mcA)$.
Consider the subposets  of $\mcP$
$$
\mcP' = \big\{1, \ldots, a_{q}\} \times \{b_{q}+1, \ldots, n\big\} 
\qquad
\text{and}
\qquad
\mcP'' = \big\{a_{q}+1, \ldots, m\} \times \{1, \ldots, b_{q}\big\},
$$
with partial order inherited from $\mcP$,
and consider the induced antichains
$$
\mcA' = \big\{ (a_1, b_1) , \ldots, (a_{q-1},b_{q-1})\big\} = \mcA \cap \mcP'
\quad
\text{and}
\quad
\mcA'' = \big\{  (a_{q+1},b_{q+1}), \ldots, (a_s,b_s)\big\} = \mcA \cap \mcP''.
$$
To the posets $\mcP'$ and $\mcP''$ we associate the polynomial rings
$$
S'= \kk[x_1, \ldots, x_{a_q}, y_{b_q+1}, \ldots, y_n]
\quad \text{and} \quad
S''= \kk[x_{a_q+1}, \ldots, x_{m}, y_{1}, \ldots, y_{b_q}].
$$
They inherit a fine grading from $S$, cf. \Cref{NotationFineGrading}.
The ideals $ J_{\mcA'}\subseteq S'$  and $ J_{\mcA''}\subseteq S''$, associated to the antichains $\mcA', \mcA''$ by \Cref{DefIdealAntichain},  are Borel-fixed and radical by \Cref{PropositionBFJiffAntichain}.
\end{definition}

The cutting process produces two induced subgraphs
of the bipartite graph associated to $\mcA \subseteq \mcP$ (cf. \Cref{ExampleBipartiteGraph}).
We  illustrate \Cref{DefCuttingThreshold} and \Cref{DefCuttingProcess} with an example.

\begin{example}\label{ExampleCutting}
Let $m = n = 12, s = 7$, and  $\mcA = \big\{ (1,11),(2,10),(5,8),(7,6),(8,3),(9,2),(10,1)\big\}$.

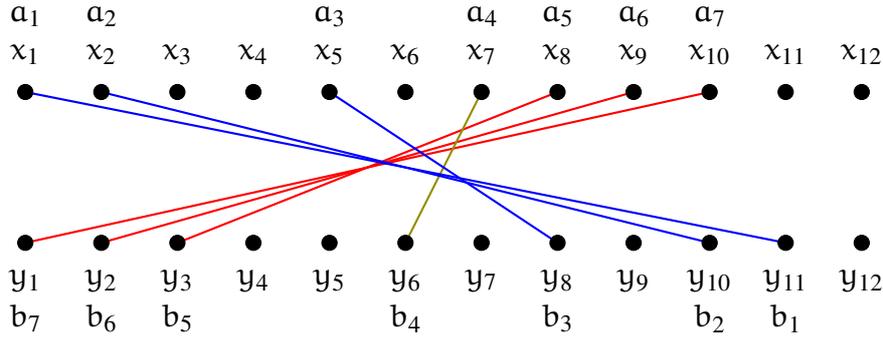
\begin{center}
\begin{tikzpicture}
\draw [red, thick] (1,0)--(10,2);
\node at (1,-1) {$b_7$};
\node at (10,3) {$a_7$};
\draw [red, thick] (2,0)--(9,2);
\node at (2,-1) {$b_6$};
\node at (9,3) {$a_6$};
\draw [red, thick] (3,0)--(8,2);
\node at (3,-1) {$b_5$};
\node at (8,3) {$a_5$};
\draw [olive, thick] (6,0)--(7,2);
\node at (6,-1) {$b_4$};
\node at (7,3) {$a_4$};
\draw [blue, thick] (8,0)--(5,2);
\node at (8,-1) {$b_3$};
\node at (5,3) {$a_3$};
\draw [blue, thick] (10,0)--(2,2);
\node at (10,-1) {$b_2$};
\node at (2,3) {$a_2$};
\draw [blue, thick] (11,0)--(1,2);
\node at (11,-1) {$b_1$};
\node at (1,3) {$a_1$};
\foreach \i in {1,...,12}
{
\draw [fill] (\i,0) circle [radius=0.1] ;
\node at (\i,-0.5) {$y_{\i}$};
\draw [fill] (\i,2) circle [radius=0.1];
\node at (\i,2.5) {$x_{\i}$};
}
\end{tikzpicture}
\end{center}
The cutting threshold of $\mcA$ is $\ct(\mcA) = q= 4$:  
\begin{enumerate}
\item  $a_s=10$ is not the last $x$-vertex $m=12$, and  $b_{q+1} =3$ is not the last $y$-vertex $n=12$,
\item the last $s-q+1$ values of $a_i$ are consecutive: they are $7,8,9,10$,
\item the first  $s-q$ values of $b_i$ are consecutive and start at 1: they are $ 1,2,3$,
\end{enumerate}
and the smaller integer $q =3$ does not satisfy some of these conditions.

The first new poset is $\mcP' =\{1,\ldots, 7\} \times \{ 7, \ldots, 12\}$.
Considering  the $q-1=3$ edges induced by the remaining vertices, 
the corresponding antichain is $\mcA' = \big\{ (1,11),(2,10),(5,8)\big\}$.
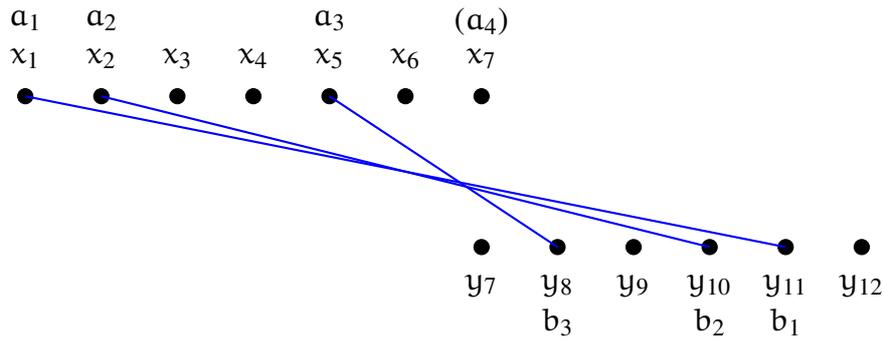
\begin{center}
 \begin{tikzpicture}
\foreach \i in {7,...,12}
{
\draw [fill] (\i,0) circle [radius=0.1] ;
\node at (\i,-0.5) {$y_{\i}$};
}
\foreach \i in {1,...,7}
{
\draw [fill] (\i,2) circle [radius=0.1];
\node at (\i,2.5) {$x_{\i}$};
}
\node at (7,3) {($a_4$)};
\draw [blue, thick] (8,0)--(5,2);
\node at (8,-1) {$b_3$};
\node at (5,3) {$a_3$};
\draw [blue, thick] (10,0)--(2,2);
\node at (10,-1) {$b_2$};
\node at (2,3) {$a_2$};
\draw [blue, thick] (11,0)--(1,2);
\node at (11,-1) {$b_1$};
\node at (1,3) {$a_1$};
\end{tikzpicture}
\end{center}

The second new poset is $\mcP'' =\{8,\ldots, 12\} \times \{1, \ldots, 6\}$.
Considering  the $s-q=3$ edges induced by the remaining vertices, 
the corresponding antichain is $\mcA'' = \big\{ (8,3),(9,2),(10,1)\big\}$.

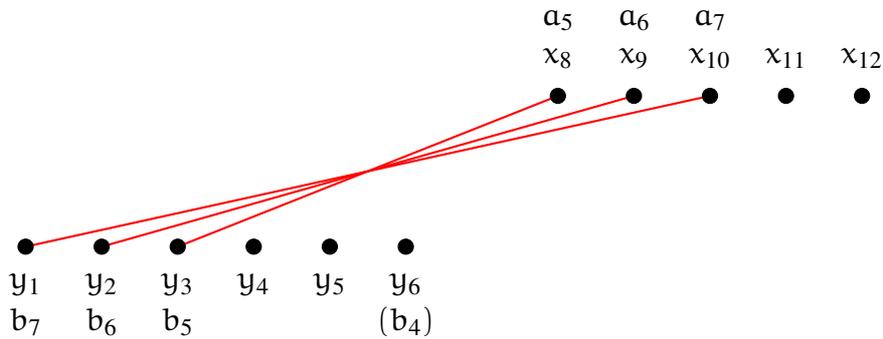
\begin{center}
\begin{tikzpicture}
\draw [red, thick] (1,0)--(10,2);
\node at (1,-1) {$b_7$};
\node at (10,3) {$a_7$};
\draw [red, thick] (2,0)--(9,2);
\node at (2,-1) {$b_6$};
\node at (9,3) {$a_6$};
\draw [red, thick] (3,0)--(8,2);
\node at (3,-1) {$b_5$};
\node at (8,3) {$a_5$};
\node at (6,-1) {$(b_4)$};
\foreach \i in {1,...,6}
{
\draw [fill] (\i,0) circle [radius=0.1] ;
\node at (\i,-0.5) {$y_{\i}$};
}
\foreach \i in {8,...,12}
{
\draw [fill] (\i,2) circle [radius=0.1];
\node at (\i,2.5) {$x_{\i}$};
}
\end{tikzpicture}
\end{center}

\end{example}

Following
 the setup of  \Cref{SectionBigradedGrassmannians},
we let $\alpha = a_q, \beta = b_q$,
then the polynomial ring $S', S''$ are the fibers of $\scR', \scR''$ over the point  $(\sfV_\circ,\sfW_\circ)\in \G$  given by 
$$
\sfV_\circ = \mathrm{Span}_\kk(x_1, \ldots, x_{a_q}), \quad \sfW_\circ = \mathrm{Span}_\kk(y_1, \ldots, y_{b_q}).
$$
Following the notation of \Cref{LemmaMapIntersection},  consider the ideals 
$$J_1 = \pi^{-1}(\iota(J_{\mcA'}))\subseteq S\quad
\text{and}
\quad  
J_2 = \pi^{-1}(\iota(J_{\mcA''})) \subseteq S.
$$

\begin{prop}\label{PropIntersectionIdealCutting}
We have
$J_\mcA = J_1 \cap J_2$ and $J_1+J_2 = (\sfV_\circ) +(\sfW_\circ)$.
\end{prop}
\begin{proof}
The intersection $J_1 \cap J_2$ is radical, since both ideals are.
It is also Borel-fixed, since both ideals are, 
as follows from the exchange property.
Thus,  $J_1 \cap J_2$ is a Cartwright--Sturmfels ideal.
Since $J_1, J_2$ do not contain  common linear forms, 
 $J_1 \cap J_2$ is generated in degree $(1,1)$ by \Cref{PropBasicPropertiesCS}.
Thus, in order to prove that $J_\mcA = J_1 \cap J_2$,
it suffices to show that, for a given $x_ay_b\in S$, we have $x_ay_b \in J_\mcA$ if and only if 
 $x_ay_b \in J_1\cap J_2$.
For simplicity, let us write the generators of the two ideals explicitly 
\begin{equation}\label{EqJ1J2Explicit}
\begin{split}
J_1 = 
\big(y_b \, : \, b \leq b_q\big) + 
\big(x_ay_b \, : \, a \leq a_q, b\geq b_q+1, (a,b) \leq (a_\ell, b_\ell) \text{ for some } \ell \leq q-1
\big),
\\
J_2 = 
\big( x_a \, :\, a \leq a_q \big)
+ \big(x_ay_b \, : \, a \geq a_q+1, b\leq b_q, (a,b) \leq (a_\ell, b_\ell) \text{ for some } \ell \geq q+1
\big).
\end{split}
\end{equation}
Suppose $a\leq a_q$.
Then, $x_ay_b \in J_2$.
We have $x_ay_b \in J_1$ if and only if $b \leq b_q$ or $b \geq b_q+1$ and
$(a,b) \leq (a_\ell, b_\ell) $ for some $\ell \leq q-1$,
equivalently, if and only if $(a,b) \leq (a_\ell, b_\ell) $ for some $\ell \leq q$.
We have $x_ay_b \in J_\mcA$ if and only if $(a,b) \in O(\mcA)$;
since  $a\leq a_q$,
 this means that $(a,b) \leq (a_\ell, b_\ell) $ for some $\ell \leq q$.
In conclusion, $x_ay_b \in J_1\cap J_2$ if and only if $x_ay_b \in J_\mcA$.
If $b \leq b_q$, we reach the same conclusion by the symmetric argument.
If $a > a_q$ and $ b > b_q$, then $x_ay_b \notin J_\mcA$
because $(a,b) \notin O(\mcA)$,
and $x_ay_b \notin J_1\cap J_2$ by \eqref{EqJ1J2Explicit}.

Finally,
the  equation $J_1+J_2 = (\sfV_\circ) +(\sfW_\circ)$ follows immediately from \eqref{EqJ1J2Explicit}.
\end{proof}

\begin{prop}\label{PropTangentDecompositionCutting}
Let $T, T'$, and  $T''$ denote the tangent spaces to $J_\mcA \in \mcH^\mfh_\kk(S),
J_{\mcA'} \in \mcH^{\mfh'}_\kk(S')$, and $J_{\mcA''} \in \mcH^{\mfh''}_\kk(S''),$ respectively.
Then,
we have
$$\dim T = \dim T' + \dim T'' + a_q(m-a_q)+b_q(n-b_q).$$ 
\end{prop}
\begin{proof}
By  \Cref{RemarkTangentVectors},
it suffices to consider linear and quadratic tangents, that is, those whose multidegree is of the form $\bfd = \bfe_j-\bfe_i,$ $\bfd = \bff_h-\bff_k,$ or $\bfd =
\bfe_j-\bfe_i + \bff_h-\bff_k$.

By \Cref{PropLinearTangents},
the linear tangents  of $T$ with multidegrees of the form  $\bfd = \bfe_j-\bfe_i$ are in bijection with pairs $i<j$ such that $i \leq a_\ell < j$ for some $\ell$.
We  divide these pairs into three sets:
\begin{enumerate}
\item Pairs such that  $i \leq a_\ell < j\leq a_q$ for some $\ell \leq q-1$.
They are in bijection with the linear tangents  of $T'$ with multidegrees of the form  $\bfd = \bfe_j-\bfe_i$.
\item Pairs such that $a_q < i \leq a_\ell < j$ for some $\ell \geq q+1$.
They are in bijection with the linear tangents  of $T''$ with multidegrees of the form  $\bfd = \bfe_j-\bfe_i$.
\item Pairs such that  $i \leq a_q  < j$. 
There are exactly $a_q(m-a_q)$ of them.
\end{enumerate}
The same analysis applies to the linear tangents with multidegrees of the form  $\bfd = \bff_h-\bff_k$, where the number of tangents in the third set is now  $b_q(n-b_q)$.

By \Cref{PropQuadraticTangents},
the quadratic tangents of $T$ 
are in bijection with  the multidegrees    $\bfd 
= \bfe_j - \bfe_i + \bff_h - \bff_k$ 
where  $(i,k) = (a_\ell, b_\ell)$ for some $\ell$, 
and moreover  $a_{\ell-1}= a_{\ell}-1$,
 $b_{\ell+1}= b_{\ell}-1$,
 $a_\ell+1 \leq j \leq a_{\ell+1}$, and
 $b_\ell+1 \leq h \leq b_{\ell-1}$.
We can divide them into three sets:
\begin{enumerate}
\item 
If $\ell \leq q-1$,
these multidegrees are in bijection with the quadratic tangents of $T'$.  
\item If $\ell \geq q+1$,
these multidegrees are in bijection with the quadratic tangents of $T''$.  
\item If $\ell = q$, this set is actually empty.
Indeed, since $
\ell = q = \ct(\mcA)$,
one of the following must occur by \Cref{DefCuttingThreshold}:
$a_\ell =m, b_\ell = n, a_\ell > a_{\ell-1}+1$, or $b_\ell > b_{\ell+1}+1$.
It follows that the conditions of \Cref{PropQuadraticTangents} are not satisfied for any $j,h$.
\end{enumerate}
In the preceding bijections, we use the endpoint convention of \Cref{NotationAntichain} for the induced antichains relative to
$\mathcal P'$ and $\mathcal P''$. In particular, the boundary vertices for $\mathcal A'$ are $a'_0=0$, $a'_q=a_q$, $b'_0=n$, and $b'_q=b_q$, while the
boundary vertices for $\mathcal A''$ are $a''_0=a_q$, $a''_{s-q+1}=m$, $b''_0=b_q$, and $b''_{s-q+1}=0$. This accounts for the boundary cases $\ell=q-1$ in (1) and $\ell=q+1$ in (2), respectively. Finally, adding the contributions of linear and quadratic tangents, we obtain the desired formula.
\end{proof}

With all this in place, we are finally prepared to prove the main result of this paper.

\begin{proof}[Proof of \Cref{MainTheorem}]
We prove the statement by induction on $n+m$, the case $n+m=0$ being trivial.
Since $\mcH^\mfh_\kk(S)$ is connected by \Cref{LemmaBFDegeneration},
it suffices to show that $\mcH^\mfh_\kk(S)$ is smooth.

By  \Cref{PropSmoothConcentratedPositiveDegrees} and \Cref{PropBasicPropertiesCS} (3), we may assume that $\mcH^\mfh_\kk(S)$ parameterizes ideals concentrated in positive degrees.
By  \Cref{PropBasicPropertiesCS} (1), this implies that every ideal  parameterized by $\mcH^\mfh_\kk(S)$ is generated in degree $(1,1)$.
By  \Cref{PropositionBFJiffAntichain}, the unique Borel-fixed ideal of 
$\mcH^\mfh_\kk(S)$ is of the form $J=J_\mcA$ for some antichain $\mcA$, cf. 
\Cref{NotationAntichain} and
\Cref{DefIdealAntichain}.
If $s=0$, that is, if $\mcA = \emptyset$, then $J_\mcA = 0$,
and it follows that $\mcH^\mfh_\kk(S)$ is a reduced point.
Thus, we assume that $s>0$, and 
let $q= \ct(\mcA)$.

Suppose first that $q>0$.
We  apply the cutting process (\Cref{DefCuttingProcess}).
Let $J' := J_{\mcA'}\subseteq S'$ and $J'' := J_{\mcA''}\subseteq S''$, and let $\mfh',\mfh''$ be the respective Hilbert functions. Also, consider the associated Cartwright--Sturmfels Hilbert schemes  $\mcH^{\mfh'}_\kk(S')$ and $\mcH^{\mfh''}_\kk(S'')$.

By 	\Cref{PropIntersectionIdealCutting},
it follows that the Hilbert functions $\mfh, \mfh', \mfh''$ are related as in \Cref{LemmaMapIntersection}.
Therefore, 
we can consider the map of 
\Cref{PropProjectionFinite}
\begin{equation}\label{EqMorphismIntersection}
\mcH_\G^{\mfh'}(\scR') \times_\G \mcH_\G^{\mfh''}(\scR'') \to \mcH_\kk^\mfh(S).
\end{equation}
First, we verify that the assumptions of \Cref{PropProjectionFinite} are satisfied.
We must show that,
 for every $\kk$-point 
$(I',I'',\sfV,\sfW)\in  \mcH_\G^{\mfh'}(\scR') \times_{\G} \mcH_\G^{\mfh''}(\scR'') $, the ideals $I',I''$ are radical and concentrated in positive degrees, and that $(\sfV)$ is not a minimal prime of $I'$, and  $(\sfW)$ is not a minimal prime of $I''$.
For this purpose, we may 
assume without loss of generality that  $\sfV=\sfV_\circ, \sfW=\sfW_\circ$, 
so that $I'\subseteq S'$ and $I'' \subseteq S''$.
Since $\mcH^{\mfh'}_\kk(S')$ and $\mcH^{\mfh''}_\kk(S'')$ are Cartwright--Sturmfels Hilbert schemes,
it follows immediately that $I',I''$ are radical ideals.
By construction, $J'$ and $J''$ are concentrated in positive degrees, 
so the same holds for every ideal parametrized 
by $\mcH^{\mfh'}_\kk(S')$ and $\mcH^{\mfh''}_\kk(S'')$.

For the last claim, assume by contradiction that $(\sfV_\circ)=(x_1, \ldots, x_{a_q})$ is a minimal prime of $I'\subseteq S'= \Sym(\sfV_\circ)\otimes \Sym(\sfY/\sfW_\circ) =
\kk[x_1, \ldots, x_{a_q}, y_{b_q+1}, \ldots, y_n]
$. 
Denote  $R = \Sym(\sfY/\sfW_\circ)=\kk[ y_{b_q+1}, \ldots, y_n]$.

The minimal primes of $I'$ are bigraded.
Any minimal prime of $I'$ other than $(\sfV_\circ)$ must contain an element of degree $(0,d)$ for some $d$, otherwise it would be contained in $(\sfV_\circ)$.
Thus, after tensoring over $R$  with the field of fractions $\mathrm{Frac}(R)$, the only surviving minimal prime is $(\sfV_\circ)$.
Moreover, since $I'$ is radical, this  forces $I'$ to coincide with $(\sfV_\circ)$ in this ring extension:
$$
I' 	\,
\kk( y_{b_q+1}, \ldots, y_n)[x_1, \ldots, x_{a_q}]
=
(x_1, \ldots, x_{a_q})\,
\kk( y_{b_q+1}, \ldots, y_n)[x_1, \ldots, x_{a_q}].
$$
Thus, $(S'/I')\otimes_R \mathrm{Frac}(R)\simeq \kk( y_{b_q+1}, \ldots, y_n)$.

If we regard $S'/I'$ as a graded module over $\kk[x_1, \ldots, x_{a_q}]$,
by considering the $x$-grading only, 
each graded component is an $R$-module.
Thus, $(S'/I')\otimes_R \mathrm{Frac}(R)$
is a  graded module over $\kk[x_1, \ldots, x_{a_q}]$, concentrated in $x$-degree 0;
in particular, every graded component of $S'/I'$ of positive $x$-degree vanishes after tensoring 
with $\mathrm{Frac}(R)$.

Now,
consider the graded component of $x$-degree 1
 $$
M(I') = \bigoplus_{d\geq 0} [S'/I']_{(1,d)},
$$
and notice that $M(I')$ is a graded $R$-module 
 generated by $x_1, \ldots, x_{a_q}$.
By the previous paragraph, we have
 $
M(I')\otimes_R \mathrm{Frac}(R) = 0,
$
that is, $M(I')$ has rank zero as $R$-module. 
Since $\HF_{I'}=\HF_{J'}$, we also have  $\HF_{M(I')}=\HF_{M(J')}$, and therefore,
 since rank is detected from the Hilbert function,
 we conclude that $M(J')$ too has rank zero.
However, this is a contradiction, since $x_{a_q}$ does not appear in the generators of $J'$, so it generates a free summand of $M(J')$.
The argument works symmetrically for $I''$, since the variable $y_{b_q}$ does not appear in the generators of $J''$.
In conclusion, \eqref{EqMorphismIntersection}
 is a finite morphism of projective $\kk$-schemes by \Cref{PropProjectionFinite}.

Since $q>0$,  both $S'$ and $S''$ have fewer variables than $S$. It follows by induction that the schemes $\mcH_\kk^{\mfh'}(S')$ and $\mcH_\kk^{\mfh''}(S'')$ are smooth and irreducible.
The fiber of the structure morphism $\mcH_\G^{\mfh'}(\scR') \times_\G \mcH_\G^{\mfh''}(\scR'') \to \G$ over an arbitrary $(\sfV,\sfW) \in \G$ 
is isomorphic to $\mcH_\kk^{\mfh'}(S') \times_\kk \mcH_\kk^{\mfh''}(S'')$  (see \Cref{RemarkBaseChange}).
By the same argument as in the proof of \Cref{PropSmoothConcentratedPositiveDegrees},
it follows that the  source scheme in \eqref{EqMorphismIntersection} is smooth and irreducible,  with dimension given by
\begin{equation}\label{EqDimensionSource}
\dim \big( \mcH_\G^{\mfh'}(\scR') \times_\G \mcH_\G^{\mfh''}(\scR'') \big)= 
\dim \mcH_\kk^{\mfh'}(S') + \dim \mcH_\kk^{\mfh''}(S'')+\dim \G.
\end{equation}
Since the morphism \eqref{EqMorphismIntersection} is finite, 
the quantity \eqref{EqDimensionSource} is a lower bound for the local dimension at some point 
of $\mcH_\kk^\mfh(S)$. 
Since local dimension is upper semicontinuous, 
it follows by  \Cref{LemmaBFDegeneration}
that this is a lower bound for the local dimension of 
$\mcH_\kk^\mfh(S)$
at the Borel-fixed point $J$.

Since each of the schemes on the right hand side of \eqref{EqDimensionSource} is smooth and irreducible,  their dimensions  agree with the dimensions of the tangent space at any point.
Thus, by  \Cref{PropTangentDecompositionCutting}, 
the quantity  \eqref{EqDimensionSource}  is also equal to the dimension of the tangent space to $\mcH_\kk^\mfh(S)$ at $J$.
This proves that  $J$ is a smooth point, and,
by \Cref{LemmaBFDegeneration} again, 
it follows that $\mcH_\kk^\mfh(S)$ is smooth.

Now suppose that $q=0$. In this case, note that $\mcP = \mcP''$. We will now construct two explicit parameterizations of $\mcH_\kk^\mfh(S)$: one for when $s \geq 2$ and one for when $s=1$.

Assume that $s\geq 2$.
It follows from 
\Cref{DefCuttingThreshold}
that $a_i = b_{s+1-i}= i $ for $i = 1, \ldots, s$,
and that
$m \geq a_s+1 = s+1$ 
and $n \geq b_1+1 =s+1$.
Pick subspaces $\sfV \subseteq \sfX= \mathrm{Span}_\kk( x_1, \ldots, x_m)$ and 
$\sfW \subseteq \sfY= \mathrm{Span}_\kk( y_1, \ldots, y_n)$
with $\dim \sfV = \dim \sfW = s+1$,
and ordered bases $(f_1, \ldots, f_{s+1})$ of $\sfV$ and $(g_1, \ldots, g_{s+1})$ of $\sfW$. 
Consider the ideal of $2$-minors of the $2\times(s+1)$ matrix formed by these two rows
\begin{equation}\label{EqIdeal2Minors}
I_2\begin{pmatrix}
f_1 & \cdots & f_{s+1}\\
g_1 & \cdots & g_{s+1}
\end{pmatrix}.
\end{equation}
It follows from 
\cite[Proof of Corollary 3.6]{CartwrightSturmfels}
that the ideal \eqref{EqIdeal2Minors} has the same Hilbert function as $J_\mcA$.

We can relativize this construction and obtain a morphism to the Hilbert scheme.
Denote by $\Pr(s+1,\sfX)$  the principal $\PGL(s+1)$-bundle on the Grassmannian $\Gr(s+1,\sfX)$
whose fiber over a point of  $\Gr(s+1,\sfX)$, parametrizing a subspace $\sfV\subseteq \sfX$, is the set of ordered bases of $\sfV$ modulo global scaling.
This bundle is constructed as the quotient of the frame bundle of the universal 
rank $s+1$ subbundle of  $\Gr(s+1,\sfX)$, 
whose fiber over $\sfV$ is the set of all ordered bases, by the scaling action of $\kk^*\subseteq \GL(s+1)$.
Define $\Pr(s+1,\sfY)$ analogously. 
As $\kk$-schemes, 
\begin{align*}
\dim \Pr(s+1,\sfX) & = \dim \Gr(s+1,\sfX)+\dim \PGL(s+1) = (m-s-1)(s+1)+(s+1)^2-1\\
\dim \Pr(s+1,\sfY) &= \dim \Gr(s+1,\sfY)+\dim \PGL(s+1) = (n-s-1)(s+1)+(s+1)^2-1.
\end{align*}
Over the product $B = \Pr(s+1,\sfX) \times \Pr(s+1,\sfY)$
we have a universal $2\times(s+1)$ matrix of relative linear forms on $\Spec(S)\times B$. 
Let $\mathcal I \subseteq S\otimes_\kk \mathcal O_{B}$ be the ideal generated by its $2$-minors. 
For every $\kk$-point $b\in B$, the
fiber $\mathcal I_b$ is an ideal of the form \eqref{EqIdeal2Minors}, and therefore has Hilbert function $\mfh$. 
Since this is a finitely generated
graded family with constant Hilbert function, $(S\otimes_\kk \mathcal O_B)/\mathcal I$ is flat over $B$. Thus, $\mathcal I$ defines a morphism
\begin{equation}\label{EqMorphismGrPGL}
\Pr(s+1,\sfX) \times \Pr(s+1,\sfY)
\longrightarrow \mcH^\mfh_\kk(S).
\end{equation}

The ideal \eqref{EqIdeal2Minors} recovers the subspaces $\sfV, \sfW$ uniquely,
since its ideal of partial derivatives is generated by $\sfV+\sfW$.
It recovers the two ordered bases up to a  $\mathrm{PGL}(s+1)$ embedded diagonally in $\mathrm{PGL}(s+1)^2$, cf. \cite[top of p. 1748]{CartwrightSturmfels}.
It follows that the fiber of \eqref{EqMorphismGrPGL} 
over each point in the image has dimension $(s+1)^2-1$,
and therefore the dimension of $\mcH^\mfh_\kk(S)$ at its Borel-fixed point $J_\mcA$ is at least
\begin{equation}\label{EqDimImage}
 (m-s-1)(s+1) + (n-s-1)(s+1)+ (s+1)^2-1.
\end{equation}
Now we compute the dimension of the tangent space at $J_\mcA$.
By \Cref{PropLinearTangents},
the number  of linear tangents with multidegree of the form $\bfd = 	\bfe_j-\bfe_i$ is equal to the number of pairs $i<j$ such that $i \leq a_\ell<j$ for some $\ell$,
equivalently, in this case,
such that $i <j$ and $i \leq s$.
We obtain  $\sum_{i=1}^{s}(m-i)=sm - {s+1 \choose 2}$  linear tangents.
Likewise, there are  $\sum_{k=1}^{s}(n-k)=sn - {s+1 \choose 2}$ linear tangents 
with multidegree of the form $\bfd = 	\bff_j-\bff_i$.
By \Cref{PropQuadraticTangents},
for each $\ell = 1, \ldots, s$ the number of quadratic tangents with $(i,k) = (a_\ell,b_\ell)$ is $(a_{\ell+1}-a_\ell)(b_{\ell-1}-b_\ell)$.
This number is $n-s$ if $\ell = 1$, 1 if $\ell = 2, \ldots, s-1$, and $m-s$ if $\ell = s$.
In conclusion, the tangent space has dimension
$$
sm - {s+1 \choose 2}+sn - {s+1 \choose 2}+(n-s)+s-2+(m-s).
$$
This number agrees with \eqref{EqDimImage}, and the conclusion follows as in the first part of the proof.

Finally, assume that  $s=1$.
We have $J_\mcA = (x_1y_1)$, so the Hilbert scheme parameterizes ideals generated by a bigraded quadric $f \in [S]_{(1,1)} \simeq \sfX \otimes \sfY$,
that is,  $\mcH^\mfh_\kk(S)\simeq \mathbb{P}(\sfX \otimes \sfY) $,
and the proof is concluded.
\end{proof}

\begin{remark}\label{RemarkSmallSchemes}
The work \cite{EGHP} served as a significant inspiration for the proof of our main theorem. 
Specifically, 
the algebraic subsets of $\mathbb{P}^{n+m-1}$ defined by our ideals  are small schemes.
By \cite[Theorem 0.4]{EGHP}, this implies that their irreducible components are linearly joined;
this condition plays a key technical role throughout our induction scheme, 
as seen in \Cref{SectionBigradedGrassmannians}, \Cref{DefCuttingThreshold} and \Cref{DefCuttingProcess}.
\end{remark}

\section{Future directions}
\label{SectionQuestions}

In this paper,  we initiated the general study of Cartwright--Sturmfels Hilbert schemes,
providing an essentially complete description in the bigraded case. 
In general, 
Cartwright--Sturmfels Hilbert schemes appear to exhibit significantly better behavior compared 
to the vast and wild class of arbitrary multigraded Hilbert schemes.
In this section, we indicate several  lines of investigation motivated by our paper.

In his seminal paper on Murphy’s law for singularities \cite{Vakil}, Vakil broadly categorized moduli spaces into two groups: those with good behavior, such as the moduli space of curves and the Hilbert schemes of arithmetically Cohen-Macaulay codimension 2 subschemes, and those with  pathological behavior, like the moduli of surfaces and  Hilbert schemes of codimension 2 subschemes, see \cite[Subsection 1.2]{Vakil}.
 We believe that  Cartwright--Sturmfels Hilbert schemes likely fall into the first category. An important first step towards making this more precise would be answering the following question.

\begin{question} 
Are Cartwright--Sturmfels Hilbert schemes  reduced?
\end{question}

Even if this is not the case, the fact that all  ideals parameterized by a Hilbert scheme are radical is quite strong, and it might restrict the possible singularities.

In analogy with \cite{SkjelnesSmith}, 
one could seek a complete classification of smooth Cartwright--Sturmfels Hilbert schemes. 
We expect that the combinatorial analysis in the case of Picard rank $r = 3$
 will be manageable, and some of our techniques may be extendable.

\begin{question} Classify the smooth Cartwright--Sturmfels Hilbert schemes.
\end{question}

As we have seen, 
the geometry of a bigraded Cartwright--Sturmfels Hilbert scheme is governed by the combinatorics of a bipartite graph. 
However, there is another notable instance of Cartwright--Sturmfels Hilbert schemes 
 depending on the combinatorics of a graph: by \cite[Theorem 3.1]{CDNG18},
the binomial edge ideal of a simple graph is Cartwright--Sturmfels.
There is a vast literature on the combinatorics of such ideals, but their behavior in moduli remains largely unexplored. 

\begin{question} 
Study Cartwright--Sturmfels Hilbert schemes of binomial edge ideals. 
For example, does the binomial edge ideal belong to a unique component?
 What is the dimension of a component containing the binomial edge ideal? 
\end{question}

Varieties with reduced flat degenerations exhibit exceptional behavior in many respects; for example, their singularities, cohomological invariants, and birational behavior are all highly constrained \cites{Brion, CV20, CDV20}. 
If the ideal defining the variety is Cartwright--Sturmfels, there are no other possible degenerations. 
Thus, from the standpoint of deformation theory, it is natural to ask the following question:

\begin{question}
Are points corresponding to prime ideals\footnote{By \cite[Proposition 3.6]{CDNG22},
being prime and Cartwright--Sturmfels is equivalent to saying that the ideal defines a multiplicity-free subvariety of $\mathbb{P}^{\mathbf{n}}$.
} on a Cartwright--Sturmfels Hilbert scheme always smooth?
\end{question}

As observed in \Cref{RemarkSmallSchemes},
bigraded Cartwright--Sturmfels ideals define small subschemes of projective space, in the sense of \cite{EGHP}, and this fact has important repercussions on the geometry of their multigraded Hilbert schemes.
A different, but related, situation is that of the locus of 2-determinantal ideals in the Grothendieck Hilbert scheme, which was investigated in \cite{RS24}.
Like Cartwright--Sturmfels ideals, these ideals also define small schemes in projective space, 
and their locus in the Hilbert scheme is well-structured.
These two examples motivate the following question.

\begin{question}
Study the deformation theory of small schemes.
For example, when does a small scheme correspond to a smooth point on the Grothendieck Hilbert scheme?
Is it possible to classify the flat degenerations among  small schemes?
\end{question}

\subsection*{Acknowledgments}
The authors would like to thank Aldo Conca, Andres Fernandez Herrero, and Allen Knutson for some helpful conversations.
They also thank the referee for reading the paper very carefully.
Computations with Macaulay2
\cite{Macaulay2} provided valuable insight during the preparation of this paper.
Ritvik Ramkumar was partially supported by NSF grant DMS-2401462, and Alessio Sammartano  was partially supported by the grants
PRIN 2020355B8Y,
PRIN 2022K48YYP,
and INdAM – GNSAGA  CUP E55F22000270001.

\end{document}